\documentclass[11pt]{article}
\usepackage{amsfonts,amsmath,amssymb}
\usepackage{amsthm}
\usepackage{fancyvrb}
\usepackage{bm}
\usepackage{graphicx}
\usepackage{tikz}
\usepackage{url}
\usepackage[ruled]{algorithm2e}
\usepackage{float}
\newtheorem{theorem}{Theorem}[section]
\newtheorem{lemma}[theorem]{Lemma}
\newtheorem{proposition}[theorem]{Proposition}
\newtheorem{corollary}[theorem]{Corollary}
\newtheorem{problem}[theorem]{Problem}
\newtheorem{characterization}[theorem]{Characterization}
\newtheorem{definition}[theorem]{Definition}
\newtheorem{remark}[theorem]{Remark}

\theoremstyle{definition}
\newtheorem{example}[theorem]{Example}

\newcommand\BShad{{\operatorname{BShad}}}

\newcommand\Corn{{\operatorname{Corn}}}

\newcommand\Shad{{\operatorname{Shad}}}
\newcommand\Borel{{\operatorname{Borel}}}
\newcommand\slex{{\operatorname{slex}}}

\newcommand\supp{{\operatorname{supp}}}

\newcommand\C{\mathcal{C}}
\def\NZQ{\mathbb}

\def\ZZ{{\NZQ Z}}

\newcommand\ds{\mathbf{ds}}

\newcommand\blfootnote[1]{%
  \begingroup
  \renewcommand\thefootnote{}\footnote{#1}%
  \addtocounter{footnote}{-1}%
  \endgroup
}

\begin{document}

\title{Computational methods for $t$-spread monomial ideals}
\author{Luca Amata\\
%{\footnotesize Department of Mathematical and Computer Sciences, Physical and Earth Sciences}\\
%{\footnotesize University of Messina}\\
{\footnotesize Ministry of Education and Merit}\\
{\footnotesize Via Tripoli 7, 98076 Sant'Agata di Militello, Messina, Italy}\\
{\footnotesize e-mail: luca.amata@posta.istruzione.it}
}  
\date{}
\maketitle

\begin{abstract}
Let $K$ be a field and $S=K[x_1,\ldots,x_n]$  a standard polynomial ring over $K$.
%In particular,
In this paper, we give new optimized algorithms to compute the smallest $t$-spread lexicographic set and the smallest $t$-spread strongly stable set containing a given set of $t$-spread monomials of $S$.
%This method induces a total order on $t$-spread strongly stable sets.
Some technical tools allowing to compute the cardinality of $t$-spread strongly stable sets avoiding their construction are also  presented.
%Further applications could be related to the study of the generic initial ideal of a $t$-spread ideal.
Some functions to ease the calculation of well known results about algebraic invariants for $t$-spread ideals are implemented in a \emph{Macaulay2} package, \texttt{TSpreadIdeals}, described in the paper.

\blfootnote{
\hspace{-0,3cm} \emph{Keywords:} squarefree monomial ideals, $t$-spread ideals, strongly stable ideals, lexicographic ideals, extremal Betti numbers, computational algebra.

\emph{2020 Mathematics Subject Classification:} 05E40, 13-04, 13B25, 13D02, 16W50, 68W30.

%* \emph{Corresponding author: Luca Amata; email: lamata@unime.it.}
}
\end{abstract}

\section{Introduction} \label{sec:1}
In this paper, we introduce some new algorithms 
%a new \emph{Macaulay2} \cite{GDS} package: \texttt{TSpreadIdeals}. The main purpose of this package is
to smoothly manage sets of $t$-spread monomials, $t$-spread ideals and some algebraic invariants of such ideals. The $t$-spread structures have been introduced by Ene, Herzog and Qureshi \cite{EHQ} in 2018. Since then some authors have investigated these new classes of ideals to generalize some known results about graded ideals of a polynomial ring (see \cite{AC8,ACF3,AEL,DHQ}, \emph{et al.}). Several %potentially generalizable
questions still remain open. So, our goal is to provide tools to simplify the future investigations of the researchers.

We also implement a new \emph{Macaulay2} \cite{GDS} package: \texttt{TSpreadIdeals}.
As will be seen below, the package contains some original results and algorithms introduced in this paper, and, other ones that have been previously analyzed by the author of this paper and some other researchers \cite{ACF2,ACF1,CAC}. To the best of our knowledge, the package \texttt{TSpreadIdeals} is the first one to contain the implementation of such techniques in a computer algebra system. 

The algorithms implemented in the package are devoted to manage $t$-spread monomials. Some auxiliary routines allow the user to check which $t$-spread class a monomial belongs to, or to sieve all the $t$-spread monomials from a list of monomials. Furthermore, there is a function giving the possibility to compute the $t$-spread shadow  of a list of monomials.
%, or to apply a shifting operator to a list of monomials. 

By means of the functions in the package \texttt{TSpreadIdeals}, it is possible to construct, in a simple way, suitable $t$-spread sets of monomials or $t$-spread ideals with particular properties. For example, given a set of $t$-spread monomials $N$, it is possible to obtain the smallest $t$-strongly stable set of monomials $B_t\{N\}$ (Definition~\ref{def:tstrongly}) that contains $N$. The same operations can be done for the smallest $t$-lex set $L_t\{N\}$ (Definition~\ref{def:tlex}). The package also provides some methods to compute \emph{a priori} the cardinality of the aforementioned sets. The theoretical justification of these methods is outlined in this paper.     

The technical functions we have mentioned allow us to provide a computational support to the characterization of important invariants of $t$-spread ideals. For instance, the problem of determining if a given configuration (a $r$-tuple of pairs of integers and a $r$-tuple of integers) represents an admissible configuration for the extremal Betti numbers of a $t$-strongly stable ideal
% both for positions and for values
(Problem~\ref{prob:Betti}, see also \cite{AC8,AC7,ACF2,ACF1}). In the case of a positive answer, it is possible to build the smallest $t$-strongly stable ideal with the given configuration of extremal Betti numbers.
% some pseudocode that are useful tools to solve a problem related to a numerical Characterization

Another supported feature of the package is related to the generalization of the \emph{Kruskal-Katona}'s theorem \cite{CAC}. The package allows one to compute the $f_t$-vector of a $t$-spread strongly stable ideal. Moreover, it is possible to state whether a sequence of integers is the $f_t$-vector of a suitable $t$-spread ideal. In the affirmative case, it is possible to build the smallest $t$-lex ideal whose $f_t$-vector coincides with the given sequence.  

From a computational point of view, the great advantage of using combinatorial methods to solve such problems is that the functions can be optimized for working faster on $t$-spread structures. Indeed, the monomials are treated as sequences of positive integers, and this allowed us to find alternative algorithms to give directly $t$-spread monomials from the computation. This means that it is possible to avoid the classical (unfortunately slow) computation involving all  monomials of the polynomial ring ($0$-spread) in order to take a quotient and obtain the desired results.

%The outline of the paper is the following.
The paper is structured in three main sections. In Section~\ref{sec:2}, to keep the paper almost self contained, we recall some basic notions that will be used throughout the paper. First, in the Subsection~\ref{sec:2_1}, the notion of $t$-spread monomial is introduced together with some of its useful properties. In the Subsection~\ref{sec:2_2}, we define particular subsets of $t$-spread monomials: $t$-spread lex and $t$-spread strongly stable sets. The Subsection~\ref{sec:2_3} is devoted to review some definitions and properties related to the extremal Betti numbers of a $t$-spread ideal.
In Section~\ref{sec:3}, we present some original computational methods to manage special sets of $t$-spread monomials. The main procedures presented here are translated in pseudocode. In Subsection~\ref{sec:3_1}, we give procedures to construct particular $t$-lex and $t$-strongly stable sets of monomials. In Subsection~\ref{sec:3_2}, some combinatorial tools allowed us to justify counting methods for the particular sets built in Subsection~\ref{sec:3_1}.
%The Subsection~\ref{sec:3_3} ends this block with some exhaustive examples related to the introduced algorithms.
Finally, Section~\ref{sec:4} contains our conclusions and perspectives.
In Appendix~\ref{sec:5}, we give a brief description of the package. The Subsection~\ref{sec:5_1} contains several \emph{Macaulay2} code examples for managing $t$-spread sets and solving the relevant problems discussed in this and other papers. The Subsection~\ref{sec:5_2} provides a list of the main methods implemented in the package together with short descriptions.
 
\section{Background and notation} \label{sec:2}

Let $S=K[x_1,\ldots,x_n]$ be the standard polynomial ring in $n$ indeterminates over a field $K$. The notions of $t$-spread monomials and  $t$-spread monomial ideals have been introduced in \cite{EHQ}. These classes of graded ideals are the objects of our investigation. 

Throughout the paper, given a positive integer $r$, we set $[r] = \{1,2,\ldots,r\}$.

\subsection{Basics on $t$-spread monomial ideals}\label{sec:2_1}
%T-SPREAD
Let $t\ge0$ be a nonnegative integer, a monomial $x_{i_1}x_{i_2}\cdots x_{i_d}$ of $S$ with $1\le i_1\le i_2\le\cdots\le i_d\le n$ is called \emph{$t$-spread}, if $i_{j+1}-i_j\ge t$, for all $j\in [d-1]$. A \emph{$t$-spread monomial ideal} is an ideal generated by $t$-spread monomials.\\
Clearly, every monomial ideal of $S$ is $0$-spread and every squarefree monomial ideal is $1$-spread. Hence, for $t\ge 1$ every $t$-spread monomial is squarefree.

The unique minimal set of monomial generators of a monomial ideal $I$ is denoted by $G(I)$. Therefore, we can define
\[G(I)_{d}= \{u\in G(I)\,:\, \deg(u)=d\}.\]

%Every monomial ideal $I$ of $S$ is a homogeneous ideal, so we can write
%\[
%I=\bigoplus_{d\ge0} I_d,
%\] 
%where $I_d=I\cap S_d$, i.e. $I_d$ is the $K$-vector space with basis the monomials of $I$ of total degree $d$. We denote by $\indeg I$ the initial degree of $I$, i.e. the minimum of the total degrees of the generators of $I$.

Let $u$ be a $t$-spread monomial of $S$; we denote by $\supp(u)$ the set of all index $i$ for which $x_i$ divides $u$, and by $\max(u)$ and $\min(u)$ the maximal and the minimal index $i$ belonging to $\supp(u)$, respectively.
%for which $x_i$ divides $u$.
By convention, we set $\max(1)=\min(1)=0$.

Let us denote by $M_{n,d,t}$ the set of all $t$-spread monomials of degree $d$ of the ring $S$. If $1+(d-1)t\leq n$ then $M_{n,d,t}$ is nonempty. Furthermore, using the notation in \cite{CAC}, we denote by $[I_j]_t$ the set of all $t$-spread monomials of degree $j$ of a monomial ideal $I$.

From \cite[Theorem 2.3]{EHQ} (see also \cite{CAC}), the cardinality of $M_{n,d,t}$ is given by 
\begin{equation}\label{eq:cardtsp}
|M_{n,d,t}|=\binom{n-(d-1)(t-1)}{d}.
\end{equation}

Now, for a nonempty subset $N$ of $M_{n,d,t}$, we define the \emph{$t$-shadow} of $N$
\begin{equation}\label{def:shad}
\Shad_t(N)= \{x_iw\, :\, w\in N\;\; \mbox{and}\;\; i=1,\ldots,n\}\cap M_{n,d+1,t}.
\end{equation}
%Of course, $\Shad_t(L)$ could be empty.

Throughout the paper, we assume that $t>0$ and $M_{n,d,t}$ is endowed with the \emph{squarefree lexicographic order}, $>_{\slex}$ \cite{AHH2}, \emph{i.e.}, let $u=x_{i_1}x_{i_2}\cdots x_{i_d}$ and  $v=x_{j_1}x_{j_2}\cdots x_{j_d}$ be two $t$-spread monomials of degree $d$, with $1\le i_1<i_2<\cdots<i_d\le n$ and $1\le j_1<j_2<\cdots<j_d\le n$, then $u>_{\slex}v$ if $i_1=j_1,\ldots,i_{s-1}=j_{s-1}$ and $i_s<j_s$, for some $1\le s\le d$.\\
By using this monomial order, if $N$ is a nonempty subset of $M_{n,d,t}$, then we denote by $\max N$ ($\min N$) the maximum (minimum) monomial of $N$ with respect to $>_{\slex}$.
%Moreover, to simplify the further reading, we write the monomials of a subset $N\subseteq M_{n,d,t}$ starting from $\max N$ until reaching $\min N$.

\subsection{Special classes of $t$-spread monomial sets}\label{sec:2_2}
% STRONLGY STABLE
Now, we recall the definitions of some interesting classes of $t$-spread monomial ideals, \emph{e.g.}, $t$-spread strongly stable ideals and $t$-spread lexicographic ideals.

\begin{definition}\label{def:tstrongly}\rm
A subset $N$ of $M_{n,d,t}$ is called a \emph{$t$-strongly stable} set if, taking a $t$-spread monomial $u\in N$, for all $j\in \supp(u)$ and all $i$, $1\le i< j$, such that $x_i(u/x_j)$ is a $t$-spread monomial, then it follows that $x_i(u/x_j)\in N$.

A $t$-spread monomial ideal $I$ is \emph{$t$-strongly stable} if $[I_j]_t$ is a $t$-spread strongly stable set for all $j$.
\end{definition}

To verify the $t$-strongly stability of a monomial ideal $I$, it is sufficient to investigate the set $G(I)$ \cite[Lemma 1.2]{EHQ}.
%In fact, let $I$ be a $t$-spread monomial ideal and suppose that for all $u\in G(I)$, for all integers $i<j$ with $j\in\supp(u)$ and such that $x_i(u/x_j)$ is a $t$-spread monomial, one has $x_i(u/x_j)\in I$. Then $I$ is a $t$-spread strongly stable ideal.

%It is clear that the notion of $t$-spread strongly stable ideal generalizes the concept of strongly stable and squarefree strongly stable ideal.\\

Let $N=\{u_1,\ldots,u_r\}\subset M_{n,d,t}$ be a set of $t$-spread monomials of $S$; we denote by $B_t\{N\}=B_t\{ u_1, \ldots, u_r\}$ the smallest $t$-strongly stable set containing $N$.
Moreover, we denote by $B_t(N)$ the $t$-strongly stable ideal generated by $B_t\{N\}$. If $N=\{u\}$ is a singleton, then we write $B_t(N)=B_t(u)$. 
%In such a case, the monomials $u_1,\ldots,u_r$ are called \emph{$t$-Borel generators}.

\begin{remark}\label{rem:ssmaxmin}\em
If $u\in M_{n,d,t}$ is a $t$-spread monomial of $S$, then
\[
\max B_t\{u\}=\max M_{n,d,t}=x_1x_{1+t}x_{1+2t}\cdots x_{1+(d-1)t}\quad \mbox{and}\quad \min B_t\{u\}=u.
\]
%Indeed, if $u=x_{i_1}x_{i_2}\cdots x_{i_d}\in M_{n,d,t}$ then by applying the Definition~\ref{def:tstrongly} at the first index of $u$ we have that $x_{1}x_{i_2}\cdots x_{i_d}\in B_t\{u\}$. Doing the same with the second index, we obtain that $x_{1}x_{1+t}\cdots x_{i_d}\in B_t\{u\}$. Thus proceeding to the last index, we have that $x_{1}x_{1+t}\cdots x_{1+(d-1)t}\in B_t\{u\}$.
\end{remark}

Given a $t$-spread monomial $v\in B_t\{u\}\subset M_{n,d,t}$, we denote with
\[
B_t[v,u]=\{w\in B_t\{u\}\, :\, v \geq_\slex w \} %>_\slex u
\]
the \emph{$t$-strongly stable segment} of initial element $v$ and final element $u$. Trivially, $B_t[u,u]=\{u\}$. In particular, we have $B_t\{u\}=B_t[\max M_{n,d,t},u]$.\\

A characterization of $t$-strongly stable ideals can be found in  \cite{JT}. For this purpose, Herzog and Hibi have introduced a partial order on $M_{n,d,t}$, the \emph{Borel order}.
Let $u=x_{i_1}x_{i_2}\cdots x_{i_d}$ and  $v=x_{j_1}x_{j_2}\cdots x_{j_d}$ be two $t$-spread monomials of degree $d$, with $1\le i_1<i_2<\cdots<i_d\le n$ and $1\le j_1<j_2<\cdots<j_d\le n$, then $v\geq_{\Borel} u$ if $j_s \leq i_s$, for $1\le s\le d$.

From \cite[Lemma~4.2.5]{JT} it follows the following characterization.

\begin{characterization}\label{char:strongly}
A set of monomials $N\subseteq M_{n,d,t}$ is $t$-strongly stable if and only if, for all $u\in N$ and all $v\in M_{n,d,t}$ such that $v \geq_\Borel u$, we have $v\in N$.
\end{characterization}

%LEX
As a particular class of $t$-strongly stable ideals, we recall the definition of $t$-spread lexicographic ideals.
 
\begin{definition}\label{def:tlex}\rm
A subset $N$ of $M_{n,d,t}$ is called a \emph{$t$-lex} set if, for all $t$-spread monomials $u\in N$ and all monomials $v\in S$ such that $v\ge_{\slex}u$, we have $v\in N$.
A $t$-spread monomial ideal $I$ is \emph{$t$-lex} if $[I_j]_t$ is a $t$-lex set for all $j$. 
\end{definition}

If $N=\{u_1,\ldots,u_r\}\subset M_{n,d,t}$ is a set of $t$-spread monomials of $S$, we denote by $L_t\{N\}=L_t\{ u_1, \ldots, u_r\}=L_t\{\min N\}=\{w\in M_{n,d,t}\, :\, w\geq_{\slex} \min N\}$, \emph{i.e.}, the smallest $t$-lex set containing $N$. Also, we denote by $L_t(N)$ the $t$-lex ideal generated by $L_t\{N\}$. As before, if $N=\{u\}$ then $L_t(N)=L_t(u)$.

As in the Remark~\ref{rem:ssmaxmin}, we can define the \emph{$t$-lex segment} of initial element $v$ and final element $u$
\[
L_t[v,u]=\{w\in L_t\{u\}\, :\, v \geq_\slex w \}. %>_\slex u
\]
Also in this case, we have $L_t[u,u]=\{u\}$ and $L_t\{u\}=L_t[\max M_{n,d,t},u]$.

\subsection{Some algebraic invariants}\label{sec:2_3}
%BETTI
Here we recall some definitions to describe important algebraic invariants of a graded ideal. The package introduced in this paper will make easy the computation of some of these invariants.
 
It is well known that every graded ideal $I$ of $S$ has a minimal graded free $S$-resolution \cite{Ei,JT},
\[
F_{\scriptscriptstyle\bullet}:0\rightarrow \bigoplus_{j\in\ZZ}S(-j)^{\beta_{r,j}}\rightarrow \cdots\rightarrow \bigoplus_{j\in\ZZ}S(-j)^{\beta_{1,j}}\rightarrow \bigoplus_{j\in\ZZ}S(-j)^{\beta_{0,j}}\rightarrow I\rightarrow 0.
\]
The integer $\beta_{i,j}$ is a graded Betti number of $I$, and  represents the dimension as a $K$-vector space of the $j$-th graded component of the $i$-th free module of the resolution. Each of the numbers $\beta_i=\sum_{j\in\ZZ}\beta_{i,j}$ is called the $i$-th Betti number of $I$.
%We set $F_i=\bigoplus_{j\in\ZZ}S(-j)^{\beta_{i,j}}$, for all $i=0,1,\ldots,s$. $F_i$ is a free finitely generated graded  $S$-module, for all $i$. Let $i\in\{0,\ldots,s\},j\in\ZZ$, the integer $\beta_{i,j}=\beta_{i,j}(I)=\dim_K\tor^S_i(K,I)_j$ is called a graded Betti number of $I$, it is the dimension as a $K$-vector space of the $j$-th graded component of the module $F_i$. Each of the numbers $\beta_i=\sum_{j\in\ZZ}\beta_{i,j}$ is called the $i$-th Betti number of $I$.

A powerful result \cite[Corollary 1.12]{EHQ} allows to easily compute the graded Betti numbers of a $t$-spread strongly stable ideal: 
\begin{equation}\label{eq:ssbetti}
\beta_{i,i+j}(I)=\sum_{u\in G(I)_j}\binom{\max(u)-t(j-1)-1}{i}.
\end{equation}

%In order to apply such a formula, we just need to know the minimal set of monomial generators of $I$. 

%\begin{remark}\rm
%If we put $t=0$ in (\ref{eq:ssbetti}), then we obtain the well-known formula of Eliahou and Kervaire \cite{EK} for the (strongly) stable ideals; whereas, if we put $t=1$ in (\ref{eqcarnum1}), then we obtain the Aramova, Herzog and Hibi formula for squarefree (strongly) stable ideals \cite{AHH2}.
%\end{remark}

%BETTI ESTREMALI
A significant subset of the graded Betti numbers is constituted by the extremal ones. The latter represent a refinement of very famous algebraic invariants: the projective dimension and the regularity of Castelnuovo-Mumford \cite{BCP,MS}.  
\begin{definition}\rm %(\cite{BCP})
A graded Betti number $\beta_{k,k+\ell}(I)\ne 0$ is called \emph{extremal} if $\beta_{i,i+j}(I)=0$ for all $i\ge k,j\ge\ell,(i,j)\ne(k,\ell)$.
\end{definition}

We report some useful results, stated in \cite{AC8}, about extremal Betti numbers. 

\begin{characterization} \label{char:exbetti} \cite[Theorem 1]{AC8}
Let $I$ be a $t$-spread strongly stable ideal of $S$.
The following conditions are equivalent:
\begin{enumerate}
\item[\em(a)] $\beta_{k,k+\ell}(I)$ is extremal;
\item[\em(b)] $k+t(\ell-1)+1=\max\big\{\max(u):u\in G(I)_\ell\big\}$ and $\max(u)<k+t(j-1)+1$, for all $j>\ell$ and for all $u\in G(I)_j$.
\end{enumerate}
\end{characterization}

\begin{corollary}\label{cor:extbetti} \cite[Corollary 2]{AC8}
Let $I$ be a $t$-spread strongly stable ideal of $S$ and let $\beta_{k,k+\ell}(I)$ be an extremal Betti number of $I$. Then
\[
\beta_{k,k+\ell}(I)=\Big|\Big\{ u\in G(I)_\ell:\max(u)=k+t(\ell-1)+1 \Big\}\Big|.
\]
\end{corollary}
	
Let $\beta_{k,k+\ell}(I)$ be an extremal Betti number of $I$, the pair $(k,\ell)$ is called a \emph{corner} of $I$.
If $(k_1,\ell_1),$ $\ldots,$ $(k_r,\ell_r)$, with $n-1\ge k_1>k_2>\cdots>k_r\ge1$ and $1\le \ell_1<\ell_2<\cdots<\ell_r$, are all the corners of a graded ideal $I$ of $S$, the set
\[
\Corn(I)=\Big\{(k_1,\ell_1),(k_2,\ell_2),\ldots,(k_r,\ell_r)\Big\}
\]
is called the \emph{corner sequence} of $I$. The $r$-tuple
\[
a(I)=\big(\beta_{k_1,k_1+\ell_1}(I),\beta_{k_2,k_2+\ell_2}(I),\ldots,\beta_{k_r,k_r+\ell_r}(I)\big)
\]
is called the \emph{corner values sequence} of $I$.

The Characterization~\ref{char:exbetti} induces a simple algorithmic process to find the corners of a $t$-strongly stable ideal $I$ (see \cite{AC7}). Let $I$ be generated in degrees $1\leq \ell_1 < \ell_2 < \cdots < \ell_r \leq n$.
If we set
\[m_{\ell_j} = \max\{\max(u)\,:\, u \in G(I)_{\ell_j}\},\]
for $j\in[r]$, then we can consider the following sequence associated to $I$:
\begin{equation}\label{eq:degseq1}
\widehat\ds(I) =\left( m_{\ell_1}-t(\ell_1-1)-1, %m_{\ell_2} -t( \ell_2-1)-1,
\ldots, m_{\ell_r}-t(\ell_r-1)-1 \right).
\end{equation}
From (\ref{eq:degseq1}) we can construct a suitable subsequence of $\widehat\ds(I)$ that we call the \emph{degree-sequence} of $I$:
\begin{equation}\label{eq:degseq2}
\ds(I)=\left( m_{\ell_{i_1}}-t(\ell_{i_1}-1)-1, %m_{\ell_{i_2}}-t(\ell_{i_2}-1)-1,
\ldots, m_{\ell_{i_q}}-t(\ell_{i_q}-1)-1 \right),
\end{equation}
with $\ell_1\leq \ell_{i_1} < \ell_{i_2}< \cdots <\ell_{i_q}= \ell_r$, and such that $\beta_{m_{\ell_{i_j}}-\ell_{i_j},\,m_{\ell_{i_j}}}(I)$ is an extremal Betti number of $I$, for $j\in [q]$.
The integer $q\le r$ is the number of the extremal Betti numbers of the $t$-stable ideal $I$.

Some numerical characterizations of the extremal Betti numbers of a $t$-spread strongly stable ideal have been given in \cite{ACF2,ACF1}. In particular, the following problem about the extremal Betti numbers of a $t$-strongly stable ideal has been solved in \cite[Theorem~4.4]{ACF2}.
  
\begin{problem}\label{prob:Betti}
Given three positive integers $n,r,t$, $r$ positive integers $a_1,\ldots,a_r$ and $r$ positive pairs of integers $(k_1,\ell_1),\ldots,(k_r,\ell_r)$, under which conditions there exists a $t$-spread strongly stable ideal $I$ of $S=K[x_1,\ldots,x_n]$ such that
\[
\beta_{k_1,k_1+\ell_1}(I)=a_1,\ldots,\beta_{k_r,k_r+\ell_r}(I)=a_r
\] 
are its extremal Betti numbers?
\end{problem}

The constructive methods used in the aforementioned papers to solve this problem have also been implemented in the package \texttt{TSpreadIdeals}.

\section{Computational aspects} \label{sec:3}

In this Section, we present some original algorithms to manage sets of $t$-spread monomials belonging to the special classes above defined. The correctness of these methods is theoretically proved.  

First, we show the simple construction of the $t$-shadow of a $t$-spread monomial. Then, we analyze the construction of $t$-strongly stable and $t$-lex sets of monomials. The procedures to manage $t$-lex sets are more simpler than the ones used to manage $t$-strongly stable sets. Thus, we illustrate some methods about $t$-lex sets, and then we will generalize them to the larger class of sets.

Moreover, we analyze some methods for counting (avoiding the construction) the monomials in $t$-lex and $t$-strongly stable sets.
%This methods are based on the particular indexes structure of the monomials belong to a specific set.

Finally, we present some exhaustive examples about the construction and counting of monomials both for $t$-lex and $t$-strongly stable sets.

Our approach is purely combinatorial, and all the theoretical results are also translated into pseudocode. 

\subsection{Construction algorithms}\label{sec:3_1}
We start this subsection by showing the construction of the $t$-shadow of a set of $t$-spread monomials (Definition~\ref{def:shad}). First, we can reduce the problem to the simpler one of calculating the $t$-shadow of a $t$-spread monomial. Indeed, $\Shad_t(u_1,\ldots,u_r)=\bigcup_{i=1}^{r}{\Shad_t(u_i)}$.
Let $u=x_{i_1}x_{i_2}\cdots x_{i_d}\in M_{n,d,t}$. We observe that applying the definition to $u$ means working on $0$-spread monomials and then make the intersection with $M_{n,d+1,t}$. Our aim is to find a way to directly obtain $t$-spread monomials.

Let us consider the support of $u$: $\{i_1, i_2, \ldots, i_d\}$. Now, we replace each index in $\supp(u)$, $i_q$, with the two values $i_q-t$ and $i_q+t$, still preserving the positions. So, we have the following list $(i_1-t, i_1+t, i_2-t, i_2+t, \ldots, i_d-t, i_d+t)$. We insert $1$ before the first element and $n$ after the last element of the list. Hence, we have
\begin{equation}\label{eq:tshad1}
(1, i_1-t, i_1+t, i_2-t, i_2+t, \ldots, i_d-t, i_d+t, n).
\end{equation}
From this procedure, we are sure that the list in (\ref{eq:tshad1}) has an even number of elements: $2(d+1)$. %Taking this element in pairs, we can think of them as $d+1$ sets containing the indexes we need to obtain the $\Shad_t(u)$:   
Let us consider the following sets:
\begin{equation*}
[1, i_1-t], [i_1+t, i_2-t], [i_2+t,i_3-t], \ldots, [i_{d-1}+t, i_d-t], [i_d+t, n],
\end{equation*}
%with the following conventions:
where
$[r,r]=\{r\}$, $[q,r]=\{q,q+1,q+2,\ldots,r-1,r\}$ if $q<r$ and $[q,r]=\emptyset$ if $q>r$. One can observe that they are $d+1$ sets containing the indexes we need to obtain the $\Shad_t(u)$. 
To clarify the notation, let us define
\begin{equation}\label{eq:tshad2}
\mathcal{I}=[1, i_1-t] \cup [i_1+t, i_2-t] \cup [i_2+t,i_3-t] \cup \cdots \cup [i_{d-1}+t, i_d-t] \cup [i_d+t, n].
\end{equation}
Hence, we can write $\Shad_t(u)=\{ux_h\ :\ h\in \mathcal{I}\}$.

\begin{example}
Let $S=K[x_1,\ldots,x_{16}]$ and $u=x_2x_5x_9x_{14}\in M_{12,4,2}$. We obtain the list of sets of indexes as in (\ref{eq:tshad2}):
\[
[1, 0]=\emptyset,\ [4, 3]=\emptyset,\ [7,7]=\{7\},\ [11, 12]=\{11,12\},\ [16, 16]=\{16\}.
\]
So, $ \mathcal{I}=\{7,11,12,16\}$, and, in order to get $\Shad_t(u)$, we just need to multiply $u$ by $\{x_7, x_{11}, x_{12}, x_{16}\}$, thus obtaining
\[
\Shad_t(x_2x_5x_9x_{14})=\{x_2x_5x_7x_9x_{14},\ x_2x_5x_9x_{11}x_{14},\ x_2x_5x_9x_{12}x_{14},\ x_2x_5x_9x_{14}x_{16}\}.
\]
\end{example}

The pseudocode in Algorithm~\ref{alg:tshad} is the algorithm for computing the $t$-shadow of a $t$-spread monomial. 

%\begin{center}
%\scalebox{0.9}{
\begin{algorithm}%[H]
\scriptsize
\caption{Computation of the $\Shad_t(u)$ in $M_{n,d+1,t}$}
\label{alg:tshad}
\KwIn{Polynomial ring $S$, monomial $u$, positive integer $t$}
\KwOut{list of monomials $shad$}
\SetKw{Error}{error}
\Begin{
   \eIf{\texttt{isTSpread($u$,$t$)}}{
      $n \gets $ number of the variables of $S$\;
      $d \gets \deg(u)$\;
      $ind \gets \{1\}$\;  
      \For{$q \gets 1$ \KwTo $d$}{
         $ind \gets ind \cup \{i_q-t, i_q+t\}$\;
         $q \gets q+1$\;			
      }
      $ind \gets ind \cup \{n\}$\;
      $shad \gets \{\}$\;
      $r \gets 1$\;
      \While{$r < 2*(d+1)$}{
         \For{$q \gets ind(r)$ \KwTo $ind(r+1)$}{
            $shad \gets shad \cup \{u*x_q\}$\;			
         }
         $r \gets r+2$\;
      }
   }{
      \Error{expected a $t$-spread monomial}\;      
   }  
\Return $shad$\;
}
\end{algorithm}
%}
%\end{center}

Now, we pass to the computation of some particular sets of $t$-spread monomials: $t$-strongly stable and $t$-lex sets.
To make the reasoning as simple as possible, we will describe the case of the computation of the $t$-lex set generated by a $t$-spread monomial. To do this, we simply show the method to compute the \emph{$t$-lex successor} of a $t$-spread monomial, if such monomial exists.

%\begin{Disc}\rm
%Let $u=x_{i_1}x_{i_2}\cdots x_{i_d}\in M_{n,d,t}$, the function to compute the \emph{$t$-lex successor} of $u$
%
%\end{Disc}

The Proposition~\ref{prop:tnextlex} is a rearrangement of the one in  \cite[Proposition~3.9]{ACF2}. The proof is adapted in order to make the algorithm construction clearer. %We can observe that a substantial difference lies in the treatment reserved to the last indeterminate of the monomial.
The following result allows to determine the \emph{$t$-lex successor} of a $t$-spread monomial $u$, \emph{i.e.}, the greatest $t$-spread monomial less than $u$.

\begin{proposition}\label{prop:tnextlex}
Let $n,d,t$ be three positive integers such that $1+(d-1)t\le n$. Let $u=x_{i_1}x_{i_2}\cdots x_{i_d}\in M_{n,d,t}$ be a $t$-spread monomial of $S$.
\begin{itemize}
\item[] Search the maximum index $q\in[d]$ such that $i_q+1\leq n-(d-q)t$;
\begin{itemize}
\item[(a)] if $q$ exists then the \emph{$t$-lex successor} of $u$ is the $t$-spread monomial
\[
x_{i_1}\cdots x_{i_{q-1}}x_{i_q+1}x_{i_q+1+t}\cdots x_{i_q+1+(d-q)t}\in M_{n,d,t};
\]
\item[(b)] if $q$ does not exist then $u$ is the smallest $t$-spread monomial of $M_{n,d,t}$.
\end{itemize}
\end{itemize}
\end{proposition}
\begin{proof} Let us consider the set $F=\left\{s\in[d]\, :\, i_s+1\leq n-(d-s)t\right\}$. If $F\neq\emptyset$ then it is possible to get the maximum of $F$, that is, the case $(a)$ holds true. Let $q=\max F$. Hence, we can construct the $t$-spread monomial $w=x_{i_1}\cdots x_{i_{q-1}}x_{i_q+1}x_{i_q+1+t}\cdots x_{i_q+1+(d-q)t}$. Indeed, by hypothesis, $i_q+1\leq n-(d-q)t$, and then we have that $i_q+1+(d-q)t\leq n$.
%, and, even more so $i_q+1 < i_q+1+t < i_q+1+(d-q)t\leq n$.
Because of the maximality of $q$, it is not possible to do the same reasoning starting from the index $i_{q+1}$, hence, the monomial $w$ has the smallest indexes that allow such a construction. Moreover, the calculation $i_q+(d-s)t-i_q-(d-s-1)t=t$, for $s=q,\ldots,d$, assures that $w$ is a $t$-spread monomial of degree $d$.
The arguments made so far imply that $w$ is the greatest monomial less than $u$ in $M_{n,d,t}$, with respect to $>_\slex$ order, that is, $w$ is the $t$-lex successor of $u$.

On the other hand, if $F=\emptyset$, then it is not possible to get the maximum of $F$, that is, the case $(b)$ holds true. Hence, for all $s\in[d]$ we have that $i_s+1> n-(d-s)t$, 
%Explicitly: $i_d+1>n$, $i_{d-1}+1>n-t$, \ldots, $i_2+1>n-(d-2)t$ and $i_1+1>n-(d-1)t$.
say, $i_s+1+(d-s)t > n$. This means that we cannot replace any indeterminate with any other having a larger index. There does not exist a $t$-spread monomial smaller than $u$ with respect to $>_\slex$.
\end{proof}
%\vspace{0,3cm}

\begin{example}
Let $S=K[x_1,\ldots,x_{13}]$ and let $u=x_2x_6x_{10}x_{13}\in M_{13,4,3}$. In such a case, $q=\max F=\max\{1,2\}=2$. %($6+1\leq 13-6$).
This fact ensures the construction of the $t$-lex successor of $u$: $w=x_2x_7x_{10}x_{13}\in M_{13,4,3}$.\\
On the contrary, taking the monomial $v=x_4x_7x_{10}x_{13}\in M_{13,4,3}$,  we have that $F=\emptyset$.
%It is sufficient to observe that $4+1>13-9$ (for the first index of $v$).
Indeed, $v$ is the smallest $3$-spread monomial of $S$.
\end{example}

The procedures used in the Proposition~\ref{prop:tnextlex} guarantee the correctness of the  
Algorithm~\ref{alg:tnextlex}. We illustrate it using the same notation as in the proposition.
%The correctness of such a method is due to the aforementioned proposition.  
	
%\begin{center}
%\scalebox{0.8}{
\begin{algorithm}%[H]
\scriptsize
\caption{Computation of the $t$-lex successor of $u$ in $M_{n,d,t}$}
\label{alg:tnextlex}
\KwIn{Polynomial ring $S$, monomial $u$, positive integer $t$}
\KwOut{monomial $w$}
\SetKw{Error}{error}
\Begin{
   \eIf{\texttt{isTSpread($u$,$t$)}}{
      $m \gets $ number of the variables of $S$\;
      $q \gets \deg(u)$\;   
      \While{$i_q+1 > m$}{
         $m \gets m-t$\;
         $q \gets q-1$\;
         \If{$q<0$}{
            \Error{no monomial}\;      
         }			
      }
      $w \gets x_{i_1}*\cdots * x_{i_{q-1}}$\;
      $m \gets m+1$\;
      \While{$q \leq \deg(u)$}{
         $w \gets w*x_{m}$\;
         $m \gets m+t$\;
         $q \gets q+1$\;
      }
   }{
      \Error{expected a $t$-spread monomial}\;      
   }  
\Return $w$\;
}
\end{algorithm}
%}
%\end{center}
%\vspace{0,2cm}

\begin{remark}\rm\label{rem:tlexseg}
The method described in Algorithm~\ref{alg:tnextlex} can be useful to compute the initial $t$-lex segment generated by a monomial $u$ of $M_{n,d,t}\subset S$. To obtain this result, one has to consider the greatest $t$-spread monomial of $M_{n,d,t}$: $x_1x_{1+t}x_{1+2t}\cdots x_{1+(d-1)t}$. After which, one can use iteratively the algorithm until reaching $u$. So, the monomials built in this way, including $u$, are all the monomials belonging to $L_t\{u\}$.
More generally, as we have already seen, $L_t\{u_1,\ldots,u_r\}=L_t\{u_r\}$ when $u_r$ is the smallest monomial in the set $\{u_i\}$, $i=1,\ldots,r$. 

Furthermore, it is possible to compute the $t$-lex segment identified by two monomials (changing the starting monomial), $L_t[v,u]$. A particular case is the \emph{$t$-spread Veronese} set, $M_{n,d,t}=L_t[\max M_{n,d,t},\min M_{n,d,t}]$. In fact, it is the $t$-lex segment whose ``extrema'' are $x_1x_{1+t}x_{1+2t}\cdots x_{1+(d-1)t}$ and $x_{n-(d-1)t}\cdots x_{n-2t}x_{n-t}x_n$.
\end{remark}

Now, what has been done previously suggests, \emph{mutatis mutandis}, how to approach the computation of $t$-strongly stable sets of monomials.

Some comments are in order. Unlike the $t$-lex case, the construction of the monomials in a $t$-strongly stable set depends on both the starting monomial and the final one. So, it is not possible to have a function with only one parameter corresponding to the $t$-lex successor. We believe that the best way is to tackle the problem in its most general form.

More technically, we have observed in Remark~\ref{rem:tlexseg} that in order to compute $L_t\{u\}$, $u\in M_{n,d,t}$, we start from $\max M_{n,d,t}$ in order to get $u$ by using the Algorithm~\ref{alg:tnextlex}. The monomial $u$ is used only to determine the end of the iterations. Indeed, the algorithm only exploits the structure of the monomial for which we want to find the $t$-lex successor.

In the computation of the $t$-strongly stable set $B\{u\}$, we also start from $\max M_{n,d,t}$ (see Remark~\ref{rem:tstrongseg}) to arrive at $u$ but, in such a case, the structure of the monomial $u$ has to be continuously taken into account to build all the needed monomials (Characterization~\ref{char:strongly}). For this reason, we will start analyzing the $t$-strongly stable set identified by two $t$-spread monomials (the greatest monomial and the smaller one).

More in detail, let $S=K[x_1,\ldots,x_n]$ be a polynomial ring over a field $K$, and let $N=\{u_1,\ldots,u_r\}\subset M_{n,d,t}$ be a set of $t$-spread monomials of $S$. We observe that $B_t\{N\}=\bigcup_{i=1}^{r}{B_t\{ u_i\}}$; hence, we can describe the singleton case $B_t\{u\}\subset M_{n,d,t}$ without loss of generality. Moreover, we recall that $B_t\{u\}=B_t[\max M_{n,d,t},u]=$ $B_t[x_1x_{1+t}\cdots x_{1+(d-1)t},u]$. So, we can face the problem to compute the $t$-strongly stable segment $B_t[v,u]$, where $v>_\Borel u$, to encompass all cases.

The intuitive way to face this computation is to apply the definition of $t$-strongly stability to the monomials and then to select all the $t$-spread monomials from the result. The drawback of this method is the slowness and the involvement of a large number of monomials, most of which will be discarded. Thence, there is a waste of time and resources, and this imposes limitations on the initial parameters, for instance, on the maximum of the supports of the involved monomials. 
%the choice of the polynomial ring.

So, to make the process faster we can work directly with a suitable sequence of $t$-spread monomials. This idea is similar to that used for the $t$-lex successor. In such a case, we need to find a method sending a $t$-spread monomial to the next one that belongs to the same $t$-strongly stable set. If we consider the lex order $>_\Borel$, from Remark~\ref{rem:ssmaxmin}, we can start our reasoning from the greatest $t$-spread monomial of $B_t[v,u]$, $v$, and reach step by step the smallest one, say $u$.
The effectiveness of this procedure relies on the possibility of suitably manipulating the indexes of a $t$-spread monomial.

Proposition~\ref{prop:tnextstrong} and Remark~\ref{rem:tstrongseg} solve the problem.

\begin{proposition}\label{prop:tnextstrong}
Let $n,d,t$ be three positive integers such that $1+(d-1)t\le n$. Let $v=x_{j_1}x_{j_2}\cdots x_{j_d}$ and $u=x_{i_1}x_{i_2}\cdots x_{i_d}\in M_{n,d,t}$ be two $t$-spread monomials of $S$ such that $v\neq u$ and $v\geq_\Borel u$.\\
%, \emph{i.e.}, $j_s\leq i_s$ for $s\in [d]$.\\
%The $t$-strongly stable set of $t$-spread monomials, $B_t[v,u]$ %with $w$ between $v$ and $u$
%$v>_\slex w >_\slex u$
%, is constructed by the following instructions:
Let $q\in[d]$ be the maximum index such that $j_q+1\leq i_q$. Then, the $t$-spread monomial
\[
w=x_{j_1}\cdots x_{j_{q-1}}x_{j_q+1}x_{j_q+1+t}\cdots x_{j_q+1+(d-q)t}\in M_{n,d,t}
\]
belongs to the $t$-strongly stable segment $B_t[v,u]$, and $w$ is the greatest monomial of $B_t[v,u]$, with respect to $>_\slex$, except $v$.
\end{proposition}

\begin{proof} Let $F=\left\{s\in[d]\, :\, j_s+1\leq i_s\right\}$. From the hypothesis $v>_\Borel u$, it is $F\neq\emptyset$. So, let $q=\max F\in[d]$. Under these conditions, we show that the monomial $w=x_{j_1}\cdots x_{j_{q-1}}x_{j_q+1}x_{j_q+1+t}\cdots x_{j_q+1+(d-q)t}$ exists. Indeed, from the hypotheses and starting from $j_q+1\leq i_q$, we have the following inequalities:
%$i_q+1\leq j_q\leq n$ and
%From this inequality and the properties of a $t$-spread monomial, we have
\[
\begin{array}{ccccc}
j_q+1+t & \leq & i_q+t & \leq & i_{q+1},\\
j_q+1+2t & \leq & i_q+2t & \leq & i_{q+2},\\
& & \vdots & &\\
j_q+1+(d-s)t & \leq & i_q+(d-s)t & \leq & i_{q+d-s},\\
& & \vdots & &\\
j_q+1+(d-q)t & \leq & i_q+(d-q)t & \leq & i_{d}.
\end{array}
\]

These results guarantee the existence of $w$ as $t$-spread monomial of $M_{n,d,t}$. Furthermore, comparing the indexes of $w$ with those of $u$, one has that $w\geq_\Borel u$.
%, where $>_\Borel$ is the partial order introduced by Herzog and Hibi in \cite{JT}.
From the Characterization \ref{char:strongly}, we can deduce that $w$ is a monomial of the $t$-strongly stable set generated by $u$.

From the choice of $q$, $w$ has the smallest indexes for which such a construction is possible. Hence, $w$ is the greatest monomial less than $v$ in $B_t[v,u]$, with respect to the $>_\slex$ order. The thesis holds true.
\end{proof}
%\vspace{0,3cm}

\begin{remark}\label{rem:tstrongseg}\rm
With the same notations of the Proposition~\ref{prop:tnextstrong}, the monomial $w_1=w$ is constructed to be less than $v$ and greater or equal to $u$.\\
So, in order to obtain all the monomials of the $t$-strongly stable segment $B_t[v,u]$, we can iterate the construction in the Proposition~\ref{prop:tnextstrong} by replacing the monomial $w_1$ with the monomial $v$ since $w_1\geq_\Borel u$. Hence, we can apply the proposition to $B_t[w_1,u]$. Indeed, all the hypotheses of the Proposition~\ref{prop:tnextstrong} are satisfied. This process allows to find a set of $t$-spread monomials ${w_1,w_2,\ldots,w_s}$ such that $w_i\geq_\Borel u$. From the construction of $w$, it is clear that $u$ will be obtained in this way. In such a case, when $w_s=u$ the proposition can no longer be applied. This is the end point of the iterations. Indeed, the construction of the monomials $w_i$ complies with the characterization of $t$-strongly stability. So, $B_t[v,u]= \{v,w_1,w_2,\ldots,w_s=u\}$.  
\end{remark}

The following example clarifies the calculation of a $t$-strongly stable segment identified by two monomials.

\begin{example}
Let $S=K[x_1,\ldots,x_9]$ and let $v=x_1x_5x_7$, $u=x_2x_5x_8\in M_{9,3,2}$. We want to compute $B_2[v,u]$.

Using the methods in Proposition~\ref{prop:tnextstrong}, we obtain $q_1=\max \{1,3\}=3$ and $w_1=x_1x_5x_8$. Applying the algorithm described in Remark~\ref{rem:tstrongseg}, we can repeat the procedure considering $B_t[w_1,u]$. Iterating the process, we go through the following steps:
\[
\begin{array}{ll}
q_1=\max \{1,3\}=3 &\quad-\qquad w_1=x_1x_5x_8,\\
q_2=\max \{1\}=1 &\quad-\qquad w_2=x_2x_4x_6,\\
q_3=\max \{2,3\}=3 &\quad-\qquad w_3=x_2x_4x_7,\\
q_4=\max \{2,3\}=3 &\quad-\qquad w_4=x_2x_4x_8,\\ 
q_5=\max \{2\}=2 &\quad-\qquad w_5=x_2x_5x_7,\\ 
q_6=\max \{3\}=3 &\quad-\qquad w_6=x_2x_5x_8=u.
\end{array}
\]

%We obtain $q_2=\max \{1\}=1$ and $w_2=x_2x_4x_6$. In the same way, passing to $B_t[w_2,u]$, we obtain $q_3=\max \{2,3\}=3$ and $w_3=x_2x_4x_7$. Then, for $B_t[w_3,u]$, we have $q_4=\max \{2,3\}=3$ and $w_4=x_2x_4x_8$ and, whereupon, for $B_t[w_4,u]$ we have $q_5=\max \{2\}=2$ and $w_5=x_2x_5x_7$. Finally, $B_t[w_5,u]$, we obtain $q_6=\max \{3\}=3$ and $w_6=x_2x_5x_8=u$. The algorithm ends.\\
Hence, we obtain the segment:
\[
B_2[v,u]=\{x_{1}x_{5}x_{7},\,x_{1}x_{5}x_{8},\,x_{2}x_{4}x_{6},\,x_{2}x_{4}x_{7},\,x_{2}x_{4}x_{8},\,x_{2}x_{5}x_{7},\,x_{2}x_{5}x_{8}\}.
\]
It is interesting to observe that if we consider $v=x_1x_5x_7$, $\overline u=x_2x_5x_7\in M_{11,3,2}$, we get
\[
B_2[v,\overline u]=\{x_{1}x_{5}x_{7},\,x_{2}x_{4}x_{6},\,x_{2}x_{4}x_{7},\,x_{2}x_{5}x_{7}\}.
\] 
Furthermore, if $v=x_1x_5x_7$, $\widetilde u=x_2x_4x_8\in M_{11,3,2}$, then the assumptions of the Proposition~\ref{prop:tnextstrong} are not valid. Indeed, $v \ngeq_\Borel u$, \emph{i.e.}, $v$ does not belong to $B_2\{\widetilde u\}$. So, $B_2[v,\widetilde u]=\emptyset$. 
\end{example}

%\begin{example}
%Let $S=K[x_1..x_11]$ and let $v=x_1x_3x_6x_8$, $u=x_2x_4x_6x_9\in M_{11,4,2}$. We want to compute $B_2[v,u]$.\\
%Using the methods in the Proposition~\ref{prop:tnextstrong}, we obtain $q_1=\max \{1,2,4\} =4$ and $w_1=x_1x_3x_6x_9$. Applying the algorithm described in the Remark\ref{rem:tstrongseg}, we can repeat the operations on $w_1$ and $u$ obtaining $q_2=\max\{1,2\}=2$ and $w_2=x_1x_4x_6x_8$. In the same way, starting from $w_2$ and $u$, we obtain $q_3=\max \{1,4\}=4$ and $w_3=x_1x_4x_6x_9$. Then, $q_4=\max \{1\}=1$ and $w_4=x_2x_4x_6x_8$. Finally, $q_5=\max \{4\}=4$ and $w_5=x_2x_4x_6x_9=u$. The algorithm ends.\\
%We have obtained the set:
%\[
%B_2[v,u]=\{x_{1}x_{3}x_{6}x_{8},\,x_{1}x_{3}x_{6}x_{9},\,x_{1}x_{4}x_{6}x_{8},\,x_{1}x_{4}x_{6}x_{9},\,x_{2}x_{4}x_{6}x_{8},\,x_{2}x_{4}x_{6}x_{9}\}.
%\]
%In our opinion, it is interesting to observe that if we consider  $v=x_1x_3x_6x_8$, $\overline u=x_2x_4x_6x_8\in M_{11,4,2}$ we obtain
%\[
%B_2[v,\overline u]=\{x_{1}x_{3}x_{6}x_{8},\,x_{1}x_{4}x_{6}x_{8},\,x_{2}x_{4}x_{6}x_{8}\}.
%\] 
%\end{example}

The algorithm arising from Proposition~\ref{prop:tnextstrong} and  Remark~\ref{rem:tstrongseg} is described through the pseudocode in Algorithm~\ref{alg:tnextstrong}.
	
%\begin{center}
%\scalebox{0.8}{
\begin{algorithm}%[H]
\scriptsize
\caption{Computation of the $t$-strongly stable segment $B_t[v,u] \subset M_{n,d,t}$}
\label{alg:tnextstrong}
\KwIn{Polynomial ring $S$, monomials $v,u$, positive integer $t$}
\KwOut{list of monomials $l$}
\SetKw{Error}{error}
\Begin{
   \eIf{\texttt{isTSpread($\{v,u\}$,$t$)} and $v \geq_\Borel u$}{
      $l \gets \{v\}$\;      
      \While{$w \neq u$}{
         $q \gets \deg(v)$\;   
         \While{$i_q+1 > j_q$}{
            $q \gets q-1$\;			
         }
         $w \gets x_{i_1}*\cdots * x_{i_{q-1}}$\;
         $m \gets i_q+1$\;
         \While{$q \leq \deg(v)$}{
            $w \gets w*x_{m}$\;
            $m \gets m+t$\;
            $q \gets q+1$\;
         }
         $l \gets l \cup w$\;
      }
   }{
      \Error{expected $t$-spread monomials belonging to $B_t\{u\}$}\;      
   }  
\Return $l$\;
}
\end{algorithm}
%}
%\end{center}

\subsection{Counting algorithms}\label{sec:3_2}

An interesting subject from a combinatorial point of view is to compute the cardinality of both $t$-lex and $t$-strongly stable sets.
%As already noted, we can focus on $L_t\{u\}=L_t[\max M_{n,d,t},u]\subset M_{n,d,t}$ and $B_t\{u\}=B_t[\max M_{n,d,t},u]\subset M_{n,d,t}$ without losing generality.
We will focus our attention on the sets $L_t\{u\}=L_t[\max M_{n,d,t},u]\subset M_{n,d,t}$ and $B_t\{u\}=B_t[\max M_{n,d,t},u]\subset M_{n,d,t}$.
The procedures are similar to those already used in \cite{AC7,ACF2}, \emph{i.e.}, they work by adding suitable binomial coefficients. Let us recall some arguments from the aforementioned papers to get the desired results also in this case.

\begin{lemma}\label{lem:decomp}
Let $n,q$ be positive integers such that $n\geq q$. Then
\begin{equation}\label{eq:decomp}
\binom{n}{q}=\binom{n-1}{q-1}+\binom{n-2}{q-1}+\cdots+\binom{q-1}{q-1}.
\end{equation}
\end{lemma}

\begin{remark}
Relation~\ref{eq:decomp} is an elementary decomposition of binomial coefficients. We just recall it since it is used in the sequel.
\end{remark}

As can be seen in \cite[Remark~3.6]{ACF1}, it will be useful to analyze the cardinality of the set to which the monomials belong to. We start analyzing the set $M_{n,d,t}$, and the following remark clarifies some aspects about the application of the Lemma~\ref{lem:decomp}.

\begin{remark}\label{rem:decomp}\rm
We recall that $|M_{n,d,t}|=\binom{n-(d-1)(t-1)}{d}$. Applying the formula (\ref{eq:decomp}), we obtain the following \emph{binomial decomposition}:
\begin{equation}\label{eq:Vdecomp1}
\begin{aligned}
& \binom{n-(d-1)(t-1)}{d}=\sum_{s=1}^{n-(d-1)t}{\binom{n-(d-1)(t-1)-s}{d-1}}=\\
& =\binom{n-(d-1)(t-1)-1}{d-1}+%\binom{n-(d-1)(t-1)-2}{d-1}+
\cdots+\binom{d-1}{d-1}.
\end{aligned}
\end{equation}
The decomposition in (\ref{eq:Vdecomp1}) has $n-(d-1)t$ contributions, each representing the number of $t$-spread monomials $w=x_{j_1}x_{j_2}\cdots x_{j_d}$ such that $j_1=\min(w)=s$, for $s=1,\ldots,n-(d-1)t$. The last value of $s$ is determined by the fact that $\max M_{n,d,t}=x_1x_{1+t}\cdots x_{1+(d-2)t}x_{1+(d-1)t}x_n$ and $\min M_{n,d,t}=x_{n-(d-1)t}x_{n-(d-2)t}\cdots x_{n-2t}x_{n-t}x_n$, that is, exceeding the value $n-(d-1)t$, starting from $1$, for the first index of the support it is not possible to build a $t$-spread monomial.

In a similar way, for a fixed index $s_1$ in the sum in (\ref{eq:Vdecomp1}),
%(equivalent to fixing the first indeterminate $x_{i_1}=x_{s_1})$ 
we can write the further following binomial decomposition:
\begin{equation}\label{eq:Vdecomp2}
%\resizebox{0.9\textwidth}{!}{
%\begin{math}
\begin{aligned}
& \binom{n-(d-1)(t-1)-s_1}{d-1}=\sum_{s=1}^{n-(d-1)t-s_1+1}{\binom{n-(d-1)(t-1)-s_1-s}{d-2}}\\
& =\binom{n-(d-1)(t-1)-s_1-1}{d-2}+%\binom{n-(d-1)(t-1)-s_1-2}{d-2}+
\cdots+\binom{d-2}{d-2}.
\end{aligned}
%\end{math}
%}
\end{equation}

Analogously, the binomial decomposition in (\ref{eq:Vdecomp2}) counts the number of monomials $w$ of $M_{n,d,t}$ such that $j_1=s_1$, for each $s=1,\ldots,n-(d-1)t-s_1+1$. In such a case, to analyze the last value of $s$ we can note that the greatest of such monomials with $j_1=s_1$ is $x_{s_1}x_{s_1+t}x_{s_1+2t}\cdots x_{s_1+(d-2)t}x_{s_1+(d-1)t}$ and the smallest one is $x_{s_1}x_{n-(d-2)t}\cdots x_{n-2t}x_{n-t}x_n$. Comparing the second indexes of the two monomials, we find that $j_2$ can assume the values from $s_1+t$ to $n-(d-2)t$. So, the index $s$ can assume $n-(d-2)t-(s_1+t)+1$ values, that is, $n-(d-1)t-s_1+1$.

Let us do the following position $s_{[k]}=\sum_{p=1}^{k}{s_p}=s_1+s_2+\cdots+s_k$. This notation will make more readable some formulas.

For the sake of clarity, at this step, we observe that fixing an index $s_2$, we have the possibility to count all the $t$-spread monomials with $j_1=s_1$ and $j_2=s_{[2]}+t-1$.

Again, as done in the previous case, we can fix an index $s=s_2$ in (\ref{eq:Vdecomp2}), and consider the next binomial decomposition:

\begin{equation}\label{eq:Vdecomp3}
%\resizebox{0.9\textwidth}{!}{
%\begin{math}
\begin{aligned}
& \binom{n-(d-1)(t-1)-s_{[2]}}{d-2}=\sum_{s=1}^{n-(d-1)t-s_{[2]}+2}{\binom{n-(d-1)(t-1)-s_{[2]}-s}{d-3}}\\
& =\binom{n-(d-1)(t-1)-s_{[2]}-1}{d-3}+
%\binom{n-(d-1)(t-1)-s_{[2]}-2}{d-3}+
\cdots+\binom{d-3}{d-3}.
\end{aligned}
%\end{math}
%}
\end{equation}
In this case, we are considering the monomials with $j_1=s_1$ and $j_2=s_{[2]}+t-1$. The greatest of these is $x_{s_1}x_{s_{[2]}+t-1}x_{s_{[2]}+2t-1}\cdots x_{s_{[2]}+(d-2)t-1}x_{s_{[2]}+(d-1)t-1}$ and the smallest $x_{s_1}x_{s_{[2]}+t-1}x_{n-(d-3)t}\cdots x_{n-t}x_{n}$. So, the index $s$ of (\ref{eq:Vdecomp3}) can assume $n-(d-1)t-s_1-s_2+2$ values. Finally, we note that the binomial coefficient of the decomposition with $s=s_3$ counts the number of monomials with $j_1=s_1$, $j_2=s_{[2]}+t-1$ and $j_3=s_{[3]}+2t-2$.

The procedure can be iterated for the other remaining indexes $s_3, s_4, \ldots, s_r$, with similar interpretations.
In order to make reading clearer, we show in Table~\ref{fig:corrind} some correspondences between the summation indexes and the indexes of the monomial.
\begin{table}%[H]
\scriptsize
\[
\begin{array}{|ccccc|ccccc|}
\hline
s_1,& s_2,& \ldots,& s_{d-2},& s_{d-1} & j_1,& j_2,& \ldots,& j_{d-2},& j_{d-1} \\
\hline
1,& 1,& \overset{1,\ldots,1}{\ldots},& 1,& 1 & 1,& 1+t,& \ldots,& 1+(d-3)t,& 1+(d-2)t\\
1,& 1,& \ldots,& 1,& 2 & 1,& 1+t,& \ldots,& 1+(d-3)t,& 2+(d-2)t\\
1,& 1,& \ldots,& 1,& 3 & 1,& 1+t,& \ldots,& 1+(d-3)t,& 3+(d-2)t\\
& & \vdots & & & & & \vdots & &\\
1,& 1,& \ldots,& 2,& 1 & 1,& 1+t,& \ldots,& 2+(d-3)t,& 2+(d-2)t\\
1,& 1,& \ldots,& 2,& 2 & 1,& 1+t,& \ldots,& 2+(d-3)t,& 3+(d-2)t\\
1,& 1,& \ldots,& 2,& 3 & 1,& 1+t,& \ldots,& 2+(d-3)t,& 4+(d-2)t\\
& & \vdots & & & & & \vdots & &\\
1,& 2,& \ldots,& 1,& 1 & 1,& 2+t,& \ldots,& 2+(d-3)t,& 2+(d-2)t\\
1,& 2,& \ldots,& 1,& 2 & 1,& 2+t,& \ldots,& 2+(d-3)t,& 3+(d-2)t\\
1,& 2,& \ldots,& 1,& 3 & 1,& 2+t,& \ldots,& 2+(d-3)t,& 4+(d-2)t\\
& & \vdots & & & & & \vdots & &\\
1,& 2,& \ldots,& 2,& 1 & 1,& 2+t,& \ldots,& 3+(d-3)t,& 3+(d-2)t\\
1,& 2,& \ldots,& 2,& 2 & 1,& 2+t,& \ldots,& 3+(d-3)t,& 4+(d-2)t\\
1,& 2,& \ldots,& 2,& 3 & 1,& 2+t,& \ldots,& 3+(d-3)t,& 5+(d-2)t\\
& & \vdots & & & & & \vdots & &\\
1,& 3,& \ldots,& 1,& 1 & 1,& 3+t,& \ldots,& 3+(d-3)t,& 3+(d-2)t\\
1,& 3,& \ldots,& 1,& 2 & 1,& 3+t,& \ldots,& 3+(d-3)t,& 4+(d-2)t\\
1,& 3,& \ldots,& 1,& 3 & 1,& 3+t,& \ldots,& 3+(d-3)t,& 5+(d-2)t\\
& & \vdots & & & & & \vdots & &\\
3,& 2,& \ldots,& 5,& 1 & 3,& 4+t,& \ldots,& 8+(d-3)t,& 8+(d-2)t\\
3,& 2,& \ldots,& 5,& 2 & 3,& 4+t,& \ldots,& 8+(d-3)t,& 9+(d-2)t\\
3,& 2,& \ldots,& 5,& 3 & 3,& 4+t,& \ldots,& 8+(d-3)t,& 10+(d-2)t\\
& & \vdots & & & & & \vdots & &\\
%s_1,& s_2,& \ldots,& s_{d-2},& s_{d-1} & s_1,& \displaystyle \sum_{k=1}^{2}{s_k}+t-1,& \ldots,& ???,& \displaystyle \sum_{k=1}^{d-1}{s_k}+(d-2)t-d+2\\
%& & \vdots & & & & & \vdots & &\\
\hline
\end{array}
\]
\caption{Some correspondences between $(s_1,\ldots,s_{d-1})$ and $(j_1,\ldots,j_{d-1})$.}
\label{fig:corrind}
\end{table}

In general, we have the following correspondence:
\begin{equation*}
\begin{array}{cccccc}
(s_1,& s_2,& \ldots,& s_{k},& \ldots,& s_{d-1})\\
& & & \updownarrow & &\\
(j_1,& j_2,& \ldots,& j_{k},& \ldots,& j_{d-1})\\
& & & \shortparallel & &\\
%\left(s_1, \phantom{\displaystyle \sum_{p=1}^{1}s}\right.& \displaystyle \sum_{p=1}^{2}{s_p}+t-1,& \ldots,& \displaystyle \sum_{p=1}^{k}{s_p}+(k-1)(t-1),& \ldots,& \displaystyle \left. \sum_{p=1}^{d-1}{s_p}+(d-2)(t-1)\right)\\
\left(s_{[1]}, \right.& s_{[2]}+t-1,& \ldots,& s_{[k]}+(k-1)(t-1),& \ldots,& \left. s_{[d-1]}+(d-2)(t-1)\right)
\end{array}
\end{equation*}
\end{remark}

% 2-nd VERSION
The considerations included in Remark~\ref{rem:decomp} allow to state the Theorem~\ref{thm:countlex}.
The result could be obtained using the techniques used in the proof of \cite[Theorem~3.10]{ACF2}, but we choose a different and more easily generalizable way.

\begin{theorem}\label{thm:countlex}
Let $n,d,t$ be positive integers such that $1+(d-1)t\le n$. Let $u=x_{i_1}x_{i_2}\cdots x_{i_d}\in M_{n,d,t}$ be a $t$-spread monomial of $S$. The cardinality of $L_t\{u\}$ can be presented as a sum of suitable binomial coefficients.
\end{theorem}
\begin{proof}
Our goal is to compute $|L_t\{u\}|$, \emph{i.e.}, the number of all monomials $w\in M_{n,d,t}$ such that $w\ge_{\slex}u$. Let $c=|M_{n,d,t}|$, we have $|L_t\{u\}|\leq c$. So, by Remark~\ref{rem:decomp}, we need to start the process with the following binomial decomposition:
\begin{equation}\label{eq:countlex}
\begin{aligned}
& c=\binom{n-(d-1)(t-1)}{d}=\sum_{s=1}^{n-(d-1)t}{\binom{n-(d-1)(t-1)-s}{d-1}}.\\
%& =\binom{n-(d-1)(t-1)-1}{d-1}+\binom{n-(d-1)(t-1)-2}{d-1}+\cdots+\binom{d-1}{d-1}.
\end{aligned}
\end{equation}

As clarified in the remark, the $s$-th binomial coefficient, $\binom{n-(d-1)(t-1)-s}{d-1}$, counts the number of $t$-spread monomials $w$ of degree $d$ with $\min(w)=s$. This tool will allows us to count the desired monomials. 

Let $w\in M_{n,d,t}$ such that $w=x_{j_1}x_{j_2}\cdots x_{j_d}$ and $w\geq_{\slex}u$. We observe that the monomials such that $j_1< i_1$ are greater than $u$ with respect to $>_\slex$, and they are counted by the sum of the first $i_1-1$ binomial coefficients in (\ref{eq:countlex}). Furthermore, if $j_1=i_1$, then we have to analyze successive indexes to verify if the monomial is greater than $u$ or not. This means having to carry out new binomial decompositions. Hence, if we consider the first $i_1$ binomial coefficients in (\ref{eq:countlex}), then we have improved the upper bound for the cardinality we want to compute. Thus, $|L_t\{u\}|\leq c_1$, where
\begin{equation}\label{eq:countlex1}
\begin{aligned}
c_1= & \sum_{s_1=1}^{i_1}{\binom{n-(d-1)(t-1)-s_1}{d-1}}=\\
 = & \sum_{s_1=1}^{i_1-1}{\binom{n-(d-1)(t-1)-s_1}{d-1}}+\binom{n-(d-1)(t-1)-i_1}{d-1}.
\end{aligned}
\end{equation}

As already observed, the first $i_1-1$ binomial coefficients in (\ref{eq:countlex1}) must be entirely added to compute the sought cardinality. Instead, the $i_1$-th binomial coefficient must be decomposed to be investigated using the other indexes of $u$. So, we have:
\begin{equation}\label{eq:countlex1_a}
%\resizebox{0.9\textwidth}{!}{
%\begin{math}
\begin{aligned}
& \binom{n-(d-1)(t-1)-i_1}{d-1} = \sum_{s_2=1}^{n-(d-1)t-i_1+1}{\binom{n-(d-1)(t-1)-i_1-s_2}{d-2}}\\
& = \binom{n-(d-1)(t-1)-i_1-1}{d-2}+%\binom{n-(d-1)(t-1)-i_1-2}{d-2}+
\cdots+\binom{d-2}{d-2}.
\end{aligned}
%\end{math}
%}
\end{equation}

We notice that the $s_2$-th binomial coefficient of the $i_1$-th decomposition (\ref{eq:countlex1_a}) represents the monomials with $j_1=i_1$ and $j_2=i_1+s_2+t-1$ (see Remark~\ref{rem:decomp}). Now, we have to select the  binomial coefficients needed for computing $|L_t\{u\}|$.

To solve our problem, it is necessary to find a better bound for the value of $j_2$. With a similar consideration done for $j_1$, all the monomials with $j_1=i_1$ and $j_2<i_2$ are greater than $u$. So, we must count entirely the first $i_2-i_1-t$ binomial coefficients in (\ref{eq:countlex1_a}). Indeed, $j_2=i_1+s_2+t-1<i_2$ implies $s_2<i_2-i_1-t+1$.
When $j_1=i_1$ and $j_2=i_2$, then we need further investigations. Hence, at this step, we have improved the bound for the cardinality: $|B_t\{u\}|\leq c_2$, with
\begin{equation*}%\label{eq:countlex2}
\begin{aligned}
c_2  = & \sum_{s_1=1}^{i_1-1}{\binom{n-(d-1)(t-1)-s_1}{d-1}} + \sum_{s_2=1}^{i_2-i_1-t}{\binom{n-(d-1)(t-1)-i_1-s_2}{d-2}}\\
  & + \binom{n-(d-2)(t-1)-i_2}{d-2}.
\end{aligned}
\end{equation*}

Now, we have to decompose the last binomial coefficient:
%calculation for the maximun s_2 with s_1=1:
%\overline n-(d-1)(t-1)-1-i_2+1+t-1 = \overline n-(d-1)(t-1)+t-1-i_2 = \overline n-(d-2)t+d-2-i_2 = \overline n-(d-2)(t-1)-i_2
%calculation for the maximun s_2 with s_1=i_1:
%\overline n-(d-1)(t-1)-i_1-i_2+i_1+t-1 = \overline n-(d-1)(t-1)+t-1-i_2 = \overline n-(d-2)t+d-2-i_2 = \overline n-(d-2)(t-1)-i_2
\begin{equation*}%\label{eq:countlex2_a}
\begin{aligned}
 & \binom{n-(d-2)(t-1)-i_2}{d-2}= \sum_{s_3=1}^{n-(d-2)t-i_2+1}{\binom{n-(d-2)(t-1)-i_2-s_3}{d-3}}.
\end{aligned}
\end{equation*}

As noted in Remark~\ref{rem:decomp}, these binomials count the monomials such that $j_1=i_1$, $j_2=i_2$ and $j_3$ between $i_2+t$ and $i_3$, that is, $i_3-i_2-t+1$ binomials. Hence, we have $|B_t\{u\}|\leq c_3$, with  
\begin{equation*}%\label{eq:countlex3}
\begin{aligned}
c_3  &=  \sum_{s_1=1}^{i_1-1}{\binom{n-(d-1)(t-1)-s_1}{d-1}} + \sum_{s_2=1}^{i_2-i_1-t}{\binom{n-(d-1)(t-1)-i_1-s_2}{d-2}}\\
  & + \sum_{s_3=1}^{i_3-i_2-t}{\binom{n-(d-2)(t-1)-i_2-s_3}{d-3}}+ \binom{n-(d-3)(t-1)-i_3}{d-3},
\end{aligned}
\end{equation*}
and so on, By iterating this procedure $d$ times we obtain the value of $|B_t\{u\}|$.

In general, as observed in Remark~\ref{rem:decomp}, we obtain the following bound:
\begin{equation*}\label{eq:countlexk}
\begin{aligned}
c_k  &=  \sum_{s_1=1}^{i_1-1}{\binom{n-(d-1)(t-1)-s_1}{d-1}} + \sum_{s_2=1}^{i_2-i_1-t}{\binom{n-(d-1)(t-1)-i_1-s_2}{d-2}}\\
 & + \cdots + \sum_{s_k=1}^{i_k-i_{k-1}-t}{\binom{n-(d-k+1)(t-1)-i_{k-1}-s_k}{d-k}}\\
 &+ \binom{n-(d-k)(t-1)-i_k}{d-k},
\end{aligned}
\end{equation*}
for $k\in [d]$.

So, for the value $k=d$ we obtain the desired cardinality, \emph{i.e.}, $c_{d}=|L_t\{u\}|$,
\begin{equation*}\label{eq:countlexd}
\begin{aligned}
c_d  &=  \sum_{s_1=1}^{i_1-1}{\binom{n-(d-1)(t-1)-s_1}{d-1}} + \sum_{s_2=1}^{i_2-i_1-t}{\binom{n-(d-1)(t-1)-i_1-s_2}{d-2}}\\
 & + \cdots + \sum_{s_d=1}^{i_d-i_{d-1}-t}{\binom{n-(t-1)-i_{d-1}-s_d}{0}}+ \binom{n-i_d}{0},
\end{aligned}
\end{equation*}
that counts the $t$-spread monomials $w$ of $M_{n,d,t}$ greater than or equal to $u$ with respect to $>_\slex$.

The number of the binomial coefficients involved in (\ref{eq:countlexd}) is $i_d-(d-1)t$,
\[
(i_1-1) + (i_2-i_1-t) + \cdots + (i_d-i_{d-1}-t) +1 = i_1 + \sum_{p=1}^{d-1}{(i_{p+1}-i_p-t)} = i_d-(d-1)t.
\]
\end{proof}

The next example illustrates the Theorem~\ref{thm:countlex}, that is, the counting method for $t$-lex sets. 

% 2-nd VERSION
\begin{example}\label{ex:countlex}
Let $S=K[x_1,\ldots,x_{11}]$, $t=3$ and $u=x_2x_6x_{10}\in M_{11,3,3}$. We want to compute $c_3=|L_t\{u\}|$.
%The greatest monomial of $M_{11,3,3}$, with respect to $>_{\slex}$, is $v=x_1x_4x_7$. We want to evaluate $m=|[v,u]|$. 

As done in Remark~\ref{rem:decomp}, $c=|M_{11,3,3}|=\binom{7}{3}=35$. Hence, we start considering the following binomial decomposition (Lemma~\ref{lem:decomp}):
\begin{equation}\label{eq:excountlex1}
\binom{7}{3}={\bf \binom{6}{2}}+\underline{\binom{5}{2}}+\binom{4}{2}+\binom{3}{2}+\binom{2}{2}.
\end{equation}

Since $i_1=2$, then all monomials $w\in M_{11,3,3}$ with $\min(w)\le i_1-1=1$ are greater than $u$. Hence, for the computation of $c_3=|L_t\{u\}|$ we must take into account the sum of the first binomial coefficient in (\ref{eq:excountlex1}), \emph{i.e.}, $c_1=\binom{6}{2}=15$. From here on out, we highlight in bold the binomial coefficients to be added and we underline and those ones to be decomposed.
		
Now, we consider the following binomial decomposition:
\[
\binom{5}{2}={\bf \binom{4}{1}}+\underline{\binom{3}{1}}+\binom{2}{1}+\binom{1}{1}.
\]		
Since $i_2-i_1-t=1$,  the number of all monomials with $j_1=i_1$ and $j_2 < i_2$ is $\binom{4}{1}=4$. Hence, adding the binomials found up to this point we have got $c_2=15+4=19$ monomials.
The next binomial decomposition we must consider is:
\[
\binom{3}{1}={\bf \binom{2}{0}+\binom{1}{0}}+\binom{0}{0}.
\]	
Since $i_3-i_2-t=1$, we must take into account $\binom{2}{0}=1$ monomial with $j_1=i_1$, $j_2=i_2$ and $j_3 < i_3$. 
So, we  obtain $19+1=20$ monomials of $M_{11,3,3}$ greater than $u$, and, adding the last binomial coefficient related to $u$, we have $c_3=|L_3\{x_2x_6x_{10}\}|=20+1=21$.
		
The following scheme summarizes the process of counting the monomials of $L_3\{x_2x_6x_{10}\}$:
\begin{align*}
\tbinom{7}{3}={\bf \tbinom{6}{2}}+ & \underline{\tbinom{5}{2}}+\tbinom{4}{2} + \tbinom{3}{2}+\tbinom{2}{2} \phantom{{\bf+\tbinom{1}{0}}+\tbinom{0}{0}\ \ \, }\rightarrow \boxed{15}\\
& \tbinom{5}{2}=
\begin{aligned}[t] {\bf \tbinom{4}{1}}+\underline{\tbinom{3}{1}}&+\tbinom{2}{1}+\tbinom{1}{1} \phantom{+\tbinom{0}{0}\ \ } \rightarrow \boxed{\phantom{1}4}\\
\tbinom{3}{1}&={\bf \tbinom{2}{0}+\tbinom{1}{0}}+\tbinom{0}{0} \rightarrow \boxed{\phantom{1}2}
\end{aligned}
\end{align*}
%\[
%\begin{aligned}
%&\tbinom{7}{3}=\mathbf{\tbinom{6}{2}}+ \underline{\tbinom{5}{2}}+\tbinom{4}{2} + \tbinom{3}{2}+\tbinom{2}{2} && \rightarrow \boxed{15}\\
%&\phantom{\tbinom{7}{3}=\mathbf{\tbinom{6}{2}}+\,\ }\tbinom{5}{2}=
%\mathbf{\tbinom{4}{1}}+\underline{\tbinom{3}{1}}+\tbinom{2}{1}+\tbinom{1}{1} &&\rightarrow \boxed{\phantom{1}4}\\
%& \phantom{\tbinom{7}{3}=\mathbf{\tbinom{6}{2}}+\tbinom{5}{2}=
%{\bf \tbinom{4}{1}}+\,\ }\tbinom{3}{1}={\bf \tbinom{2}{0}+\tbinom{1}{0}}+\tbinom{0}{0} &&\rightarrow \boxed{\phantom{1}2}
%\end{aligned}
%\]
The number of binomial coefficients involved in the counting is $i_d-(d-1)t=10-6=4$, and 
all the monomials of $L_3\{x_2x_6x_{10}\}$ are:
\[
\begin{aligned}
\begin{array}{r}
x_1x_4x_7, x_1x_4x_8, x_1x_4x_9, x_1x_4x_{10}, x_1x_4x_{11},\\
x_1x_5x_8, x_1x_5x_9, x_1x_5x_{10}, x_1x_5x_{11},\\
x_1x_6x_9, x_1x_6x_{10}, x_1x_6x_{11},\\
x_1x_7x_{10}, x_1x_7x_{11},\\
x_1x_8x_{11},\\
\end{array} & \rightarrow \boxed{15}\\
\begin{array}{r}
x_2x_5x_8, x_2x_5x_9, x_2x_5x_{10}, x_2x_5x_{11},\\
\end{array} & \rightarrow \boxed{\phantom{1}4}\\
\begin{array}{r}
x_2x_6x_9,\\
\end{array} & \rightarrow \boxed{\phantom{1}1}\\
\begin{array}{r}
x_2x_6x_{10}.\\
\end{array} & \rightarrow \boxed{\phantom{1}1}\\
\end{aligned}
\]
\end{example}

The Algorithm~\ref{alg:countlex} shows how to compute the cardinality of the initial $t$-lex segment generated by a monomial. The validity of the procedure is granted by the Theorem~\ref{thm:countlex}.
	
%\begin{center}
%\scalebox{0.8}{
\begin{algorithm}%[H]
\scriptsize
\caption{Computation of the cardinality $L_t\{u\} \subset M_{n,d,t}$}
\label{alg:countlex}
\KwIn{Polynomial ring $S$, monomial $u$, positive integer $t$}
\KwOut{positive integer $c$}
\SetKw{Error}{error}
\Begin{
   \eIf{\texttt{isTSpread($u$,$t$)}}{
      $n \gets$ number of indeterminates of $S$\;      
      $d \gets \deg(u)$\;
      $decomp \gets \{\}$\;
      \For{$i \gets 1$ \KwTo $n-(d-1)*t$}{
         $decomp \gets decomp \cup \big\{\{n-(d-1)*(t-1)-i, d-1\}\big\}$\;      
      }     
      $c \gets 0$\;
      $sub \gets 0$\;
      \For{$q \gets 0$ \KwTo $d-1$}{
         $s \gets 0$\;
         \If{$q>0$}{
            $sub \gets i_{q-1}+t$\;
         }        
         \While{$s < i_q-sub$}{
            $c \gets c+$ binomial of $decomp(s)$\;
            $s \gets s+1$\;
         }
         $tmp \gets \{\}$\;
         \For{$i \gets 0$ \KwTo $decomp(s)(1)-decomp(s)(2)$}{
            $tmp \gets tmp\cup \big\{\{decomp(s)(1)-i-1, decomp(s)(2)-1\}\big\}$\;      
         }
         $decomp \gets tmp$\;
      }   
   }{
      \Error{expected a $t$-spread monomial}\;      
   }    
\Return $c+1$\;
}
\end{algorithm}
%}
%\end{center}

Now, we pass to analyze the $t$-strongly stable set of monomials generated by a monomial. From Remark~\ref{rem:tstrongseg}, we have $B_t\{u\}=B_t[x_1x_{1+t}x_{1+2t}\cdots x_{1+(d-1)t},u]$. To compute $|B_t\{u\}|$ we have to count all the monomials of $M_{n,d,t}$ built by the Algorithm~\ref{alg:tnextstrong}. Moreover, Proposition~\ref{prop:tnextstrong} gives some tools to identify the desired monomials by conditions on their supports.

The following result shows an algorithmic method to find the cardinality of the set $B_t\{u\}\subset M_{n,d,t}$. The problem is more complicated than the one solved in Theorem~\ref{thm:countlex}. Nevertheless, we will use the same approach.

\begin{theorem}\label{thm:countstrong}
Let $n,d,t$ be positive integers such that $1+(d-1)t\le n$. Let $u=x_{i_1}x_{i_2}\cdots x_{i_d}\in M_{n,d,t}$ be a $t$-spread monomial of $S$. The cardinality of $B_t\{u\}$ is the sum of suitable binomial coefficients.
\end{theorem}
\begin{proof}
The inclusion $B_t\{u\}\subset M_{n,d,t}$ can be improved by noting that for each monomial $w$ of $B_t\{u\}$ we have $\max(w)\leq \max(u)$. Hence, the $t$-strongly stable set remains unchanged if we consider $B_t\{u\}\subset M_{\overline n,d,t}$, where $\overline n=\max(u)$. So, we only have to count the monomials of $M_{\overline n,d,t}$ that satisfy the conditions in Proposition~\ref{prop:tnextstrong}. If $c=|M_{\overline n,d,t}|$, then we can write $|B_t\{u\}|\leq c$.

Let $w=x_{j_1}x_{j_2}\cdots x_{j_d}$ be a monomial of $B_t\{u\}$. In such a case, each index in $\supp(w)$ is bounded by the corresponding index in $\supp(u)$, \emph{i.e.}, $j_s\leq i_s$ for $s\in[d]$ (see Characterization~\ref{char:strongly}). To count all the monomials of $B_t\{u\}$, we can properly exploit some binomial decompositions iteratively.    

Let us start by considering the binomial decomposition of $c=|M_{\overline n,d,t}|$
%\binom{\overline n-(d-1)(t-1)}{d}$
induced by the Lemma~\ref{lem:decomp}
%\small
\begin{equation}\label{eq:countstrong1}
\begin{aligned}
c= & \binom{\overline n-(d-1)(t-1)}{d}= \sum_{s=1}^{\overline n-(d-1)t}{\binom{\overline n-(d-1)(t-1)-s}{d-1}}.\\
% = & \binom{\overline n-(d-1)(t-1)-1}{d-1}+\binom{\overline n-(d-1)(t-1)-2}{d-1}+\cdots+\binom{d-1}{d-1}.
\end{aligned}
\end{equation}
%\normalsize
The meaning of such a decomposition has been analyzed in the Remark~\ref{rem:decomp}. We recall that the pivotal idea of the remark is to associate the index $j_1$ to a binomial coefficient of the decomposition (\ref{eq:countstrong1}), based on the value it  assumes.

In particular, we have to pay attention only to the first $i_1$ binomial coefficients of (\ref{eq:countstrong1}). Indeed, from $j_1\leq i_1$, we have that the first indeterminate of $w$ can assume all values between $1$ and $j_1$, \emph{i.e.}, $j_1\in[i_1]$. So, we will restrict our investigation to the following coefficients:
%\small
\begin{equation}\label{eq:countstrong1_a}
\begin{aligned}
c_1= & \sum_{s_1=1}^{i_1}{\binom{\overline n-(d-1)(t-1)-s_1}{d-1}}\\
 = & \binom{\overline n-(d-1)(t-1)-1}{d-1}+\cdots+\binom{\overline n-(d-1)(t-1)-i_1}{d-1}.
\end{aligned}
\end{equation}
%\normalsize
More in detail, the $s_1$-th contribution represents the number of the monomials with $j_1=s_1$.

Now, we have to observe that the binomial coefficients in (\ref{eq:countstrong1_a}) must not be fully added to compute $|B_t\{u\}|$. Indeed, for each of them we will consider further decompositions that will be related to the value of $j_2$. Hence, this summation provides a bound for the target cardinality: $|B_t\{u\}|\leq c_1$. So, we need to improve this bound until reaching the exact value.

As far as the investigation of the second index of $w$, $j_2$, is concerned, we must consider further binomial decompositions for each of the addends in (\ref{eq:countstrong1_a}) (see Remark~\ref{rem:decomp}). We can observe that, unlike the $t$-lex case analyzed in the Theorem~\ref{thm:countlex}, we need to decompose in parallel several binomial coefficients continuing recursively until a final condition is reached: 
%\begin{equation}\label{eq:countstrong2}
\begin{align}
& \binom{\overline n-(d-1)(t-1)-1}{d-1} =  \sum_{s_2=1}^{\overline n-(d-1)t}{\binom{\overline n-(d-1)(t-1)-1-s_2}{d-2}}=\nonumber\\
& = \binom{\overline n-(d-1)(t-1)-2}{d-2} + %\binom{\overline n-(d-1)(t-1)-3}{d-2} +
\cdots+\binom{d-2}{d-2};\label{eq:countstrong2i}\\
& \cdots\nonumber\\
& \binom{\overline n-(d-1)(t-1)-s_1}{d-1} = \sum_{s_2=1}^{\overline n-(d-1)t-s_1+1}{\binom{\overline n-(d-1)(t-1)-s_1-s_2}{d-2}}=\nonumber\\
& = \binom{\overline n-(d-1)(t-1)-s_1-1}{d-2}+%\binom{\overline n-(d-1)(t-1)-s_1-2}{d-2}+
\cdots+\binom{d-2}{d-2};\label{eq:countstrong2ii}\\
& \cdots\nonumber\\
& \binom{\overline n-(d-1)(t-1)-i_1}{d-1} = \sum_{s_2=1}^{\overline n-(d-1)t-i_1+1}{\binom{\overline n-(d-1)(t-1)-i_1-s_2}{d-2}}=\nonumber\\
& = \binom{\overline n-(d-1)(t-1)-i_1-1}{d-2}+%\binom{\overline n-(d-1)(t-1)-i_1-2}{d-2}+
\cdots+\binom{d-2}{d-2}\label{eq:countstrong2iii}.
\end{align}
%\end{equation}

To be more clear, we recall that the $s_2$-th binomial coefficient of the $s_1$-th decomposition represents the number of the monomials with $j_1=s_1$ and $j_2=s_{[2]}+t-1$. Now, let us move on to discuss about the number of binomial coefficients, related to the index $s_2$, in order to compute $|B_t\{u\}|$.

First, we consider the decomposition of the binomial coefficients in (\ref{eq:countstrong2i}). They represent all the monomials whose index $j_1=1$, and the addenda of this sum are related to the index $j_2$. This index can assume the values between $1+t$ and $i_2$, that is, we only have to consider the first $i_2-(1+t)+1$ binomial coefficients.
%This is the maximum value for $s_2$.
%, for $s_2=1,\ldots, i_2-(1+t)$.
Analogously, from the decomposition of the binomial for counting monomials with $j_1=2$, we can state that $j_2$ can assume the values between $2+t$ and $i_2$, that is, $i_2-(1+t)$ values.

In general, as we can see in (\ref{eq:countstrong2ii}), from the decomposition of the $s_1$-th component, $j_2$ assumes values between $s_1+t$ and $i_2$, that is, $i_2-(s_1+t)+1$ values.

With this in mind, we can improve the bound for the cardinality: $|B_t\{u\}|\leq c_2$, with
\begin{equation*}
\begin{aligned}
c_2 & =  \sum_{s_1=1}^{i_1}{ \sum_{s_2=1}^{i_2-(s_1+t)+1}{\binom{\overline n-(d-1)(t-1)-s_{[2]}}{d-2}}}=\\
& = \sum_{s_2=1}^{i_2-(1+t)+1}{\binom{\overline n-(d-1)(t-1)-1-s_2}{d-2}}+ \cdots \\
% \sum_{s_2=1}^{i_2-(s_1+t)+1}{\binom{\overline n-(d-1)(t-1)-s_1-s_2}{d-2}}\\
&+\sum_{s_2=1}^{i_2-(i_1+t)+1}{\binom{\overline n-(d-1)(t-1)-i_1-s_2}{d-2}},
\end{aligned}
\end{equation*}
that can be written as
%calculation for the maximun s_2 with s_1=1:
%\overline n-(d-1)(t-1)-1-i_2+1+t-1 = \overline n-(d-1)(t-1)+t-1-i_2 = \overline n-(d-2)t+d-2-i_2 = \overline n-(d-2)(t-1)-i_2
%calculation for the maximun s_2 with s_1=i_1:
%\overline n-(d-1)(t-1)-i_1-i_2+i_1+t-1 = \overline n-(d-1)(t-1)+t-1-i_2 = \overline n-(d-2)t+d-2-i_2 = \overline n-(d-2)(t-1)-i_2
\begin{equation}\label{eq:countstrong2_a}
\begin{aligned}
c_2 &= \binom{\overline n-(d-1)(t-1)-2}{d-2} + \cdots + \binom{\overline n-(d-2)(t-1)-i_2}{d-2}+\cdots\\
% & + \binom{\overline n-(d-1)(t-1)-s_1-1}{d-2}+\cdots+\binom{\overline n-(d-2)(t-1)-s_1+s_2+1}{d-2}\\
&+ \binom{\overline n-(d-1)(t-1)-i_1-1}{d-2}+\cdots+\binom{\overline n-(d-2)(t-1)-i_2}{d-2}.
\end{aligned}
\end{equation}

Now, we must consider further decompositions. So, from the first binomial coefficient of the first row of (\ref{eq:countstrong2_a}), for $j_1=1$ and $j_2=1+t$, the index $j_3$ of $w$ can assume the values from $1+2t$ to $i_3$, that is, we have $i_3-2t$ binomial coefficients. From the last coefficient of the first row of (\ref{eq:countstrong2_a}), for $j_1=1$ and $j_2=i_2$, the index $j_3$ can take the values from $i_2+t$ and $i_3$, so, $i_3-i_2-t+1$ binomial coefficients. From the first binomial of the second row, for $j_1=i_1$ and $j_2=i_1+t$, the index $j_3$ can assume the values from $i_1+2t$ and $i_3$, then $i_3-i_1-2t+1$ coefficients. Finally, from the last binomial of the second row, for $j_1=i_1$ and $j_2=i_2$, $j_3$ can take the values from $i_2+t$ and $i_3$, so, $i_3-i_2-t$ binomials.  
More in general, attempting to write a closed formula, if we fix the indexes $j_1=s_1$ and $j_2=s_{[2]}+t-1$ for the monomials $w\in B_t\{u\}$, then $j_3$ can assume the values between $s_{[2]}+2t-1$ and $i_3$, that is, $i_3-s_{[2]}-2t+2$ values.
%we can note that $s_1=1\ldots i_1$, that is, the summand index has the same value of the index of the monomial. 
And so on, for the successive decompositions. By iterating this procedure we obtain better and better bounds for $|B_t\{u\}|$ until we reach its exact value.

In general, as also observed in Remark~\ref{rem:decomp}, the index $j_k$ of $w$ may assume at most the values from $s_{[k-1]}+(k-1)t-k+2=s_{[k-1]}+(k-1)(t-1)+1$ to $i_k$, hence, we must take $i_k-s_{[k-1]}-(k-1)(t-1)$ binomial coefficients, whereupon
we can write
\begin{equation*}\label{eq:countstrong3}
\begin{aligned}
c_k & =  \sum_{s_1=1}^{i_1} \sum_{s_2=1}^{i_2-s_1-t+1}
\cdots \sum_{s_k=1}^{i_k-s_{[k-1]}-(k-1)(t-1)}{\binom{\overline n-(d-1)(t-1)-s_{[k]}}{d-k}},
\end{aligned}
\end{equation*}
for $k=1,\ldots,d-1$.

Finally, for the value $k=d-1$, we obtain the desired cardinality, \emph{i.e.}, 
\begin{equation}\label{eq:countstrong3_a}
%\resizebox{0.9\textwidth}{!}{
%\begin{math}
\begin{aligned}
%& |B_t\{u\}| = c_{d-1}\\
c_{d-1} 
& =  \sum_{s_1=1}^{i_1}
\sum_{s_2=1}^{i_2-s_{1}-t+1}
%\cdots \sum_{s_k=1}^{i_k-(s_{k-1}+(k-1)t)+1}
\cdots \sum_{s_{d-1}=1}^{i_{d-1}-s_{[d-2]}-(d-2)(t-1)} {\binom{\overline n-(d-1)(t-1)-s_{[d-1]}}{1}}.
\end{aligned}
%\end{math}
%}
\end{equation}
Indeed, the formula (\ref{eq:countstrong3_a}) counts the $t$-spread monomials $w$ of $M_{n,d,t}$ whose indexes respect the conditions of the Borel order in relation to the monomial $u$, $j_s\geq i_s$ for $s\in [d]$.
\end{proof}

The following remark goes into detail on the number of the suitable binomial coefficients mentioned in Theorem~\ref{thm:countstrong}.

\begin{remark}\rm\label{rem:operator}
Under the same hypotheses and notation of Theorem~\ref{thm:countstrong}, we can state that the cardinality $|B_t\{u\}|$ is a sum of 
\begin{equation}\label{eq:operator}
%\C_{d-1}\big(i_{d-1}-(1+(d-2)t)+1, i_{d-2}-(1+(d-3)t)+1,\ldots , i_2-(1+t)+1,i_1\big)
\C_{d-1}\big(i_{d-1}-(d-2)t, i_{d-2}-(d-3)t,\ldots , i_2-t,i_1\big)
\end{equation}
suitable binomial coefficients, where the operators $\C_q$, $q\geq 1$, are defined as follows.
Let $a_1,a_2,\ldots,a_q$ be $q$ positive integers such that $a_r\geq a_{r+1}$ for $r\in[q-1]$, then  
\begin{equation*}
\C_q(a_1,a_2,\ldots,a_q) =
\begin{cases}
a_1 & \mbox{if }\, q=1;\\
%a_1+(a_1-1)+\cdots+(a_1-a_2+1) & \mbox{if }\, k=2;\\
%\C_{k-1}(a_1,\ldots,a_{k-1}),a_k) + \cdots + \C_{k-1}(a_1-a_k+1,\ldots,a_{k-1}-a_k+1) & \mbox{if }\, k>1.
\displaystyle \sum_{r=0}^{a_q-1}{\C_{q-1}(a_1-r,\ldots,a_{q-1}-r)} & \mbox{if }\, q>1.
\end{cases}  
\end{equation*}

As  far as the formula~(\ref{eq:operator}) is concerned, in general, the argument with index $k$ is exactly $i_k-s_{[k-1]}-(k-1)(t-1)$, \emph{i.e.}, the maximum number of binomial coefficients we have considered in (\ref{eq:countstrong3_a}). Then, to compute the argument indexed with $k$ we can set $s_p=1$, for $p\in [k-1]$, and we obtain 
\[
i_k-k+1-(k-1)(t-1)=i_k-(k-1)t.
\] 
These positions allow to easily calculate \emph{a priori} the number of binomial coefficients involved in the formula~(\ref{eq:countstrong3_a}), using only the support of the given monomial $u$.
%maximum values that each index of the monomial $w\in B_t\{u\}$ can achieve.   
For example, the calculation of $\C_3(6,4,2)$ is the following:
\[
\begin{aligned}
\C_3(6,4,2)&=\C_2(6,4)+\C_2(5,3)=\\
&=\C_1(6)+\C_1(5)+\C_1(4)+\C_1(3)+\C_1(5)+\C_1(4)+\C_1(3)=30.
\end{aligned}
\]

%A property of the operators $\C_k$ is the following
%\[\C_2(a_1,a_2)???\]
\end{remark}

%ALTERNATIVE DEFINITION
%\begin{remark}\rm
%Hence, as Formula~(\ref{eq:count3_a}) shows, the cardinality $|B_t\{u\}|$ is a sum of 
%\[
%\left(i_{d-1}-(1+(d-2)t)+1\right) \downarrow\cdots\downarrow\left(i_k-(1+(k-1)t)+1\right)\downarrow\cdots \downarrow i_1
%\]
%binomial coefficients, where the operator $\downarrow$ in defined as follows. Let $a,b$ be two positive integer such that $a\geq b$, then  
%\begin{equation*}
%a\downarrow b=
%\begin{cases}
%a & \mbox{if }\, b=1;\\
%a+(a-1)\downarrow(b-1)  & \mbox{if }\, b>1.
%\end{cases}
%\end{equation*}
%
%The operator $\downarrow$ is not associative, but we will define a special left-associative law, \emph{i.e.}, $a\downarrow b\downarrow c=(a\downarrow b)\downarrow c$ defined by
% 
%\begin{equation*}
%(a\downarrow b)\downarrow c = 
%\begin{cases}
%a\downarrow b & \mbox{if }\, c=1;\\
%a\downarrow b +(a-1)\downarrow(b-1)  & \mbox{if }\, c>1.
%\end{cases}
%\end{equation*}
%The iteration of the previous definition allows to compute the general form
%\[
%a_1 \downarrow a_2 \downarrow a_3 \downarrow \cdots \downarrow a_k = (\cdots ((a_1 \downarrow a_2) \downarrow a_3) \downarrow \cdots) \downarrow a_k.
%\]
%These positions make it possible to easily calculate \emph{a priori} the numbers of binomial coefficient involved in the Formula~(\ref{eq:count3_a}).
%\end{remark}

The following two examples illustrate the procedure to compute the cardinality of $B_t\{u\}$ (see Theorem~\ref{thm:countstrong}).

\begin{example}\label{ex:countstrong1}
Let $S=K[x_1,\ldots,x_{13}]$, $t=1$ and $u=x_{i_1}x_{i_2}x_{i_3}x_{i_4}=x_2x_5x_8x_{11}\in S$. We want to compute the cardinality of $B_t\{u\}$. The greatest monomial of $S$, with respect to $>_{\slex}$, is $x_1x_2x_3x_4$, whence we need to compute $c_3=|B_1[x_1x_2x_3x_4,u]|=|B_1\{u\}|$.
		
As observed, we can limit our investigation to monomials with $\overline n=\max(u)=11$, that is, to monomials belonging to $M_{\overline n,d,t}$. The number of all squarefree monomials of $S$ of degree $4$ is $c=|M_{11,4,1}|=\binom{11}{4}=330$.
		
Let us consider the following binomial decomposition (Lemma~\ref{lem:decomp}):
\begin{equation}\label{eq:excountstrong1_0}
\binom{11}{4}=\underline{\binom{10}{3}}+\underline{\binom{9}{3}}+\binom{8}{3}+\binom{7}{3}+\binom{6}{3}+\binom{5}{3}+\binom{4}{3}+\binom{3}{3}.
\end{equation}
To count all the monomials $w=x_{j_1}x_{j_2}x_{j_3}x_{j_4}\in B_1\{u\}$, we can observe that we need to count all monomials where the first index is less than or equal to $2$, the second one is less than or equal to $5$ and so on, \emph{i.e.}, $j_1\leq 2$, $j_2\leq 5$, $j_3\leq 8$ and $j_4\leq 11$.

Looking at the first index of $u$, $i_1=2$, we must consider the first two binomial coefficients in (\ref{eq:excountstrong1_0}), \emph{i.e.}, $c_1=\binom{10}{3}+\binom{9}{3}$. Albeit this sum represents a bound for the cardinality, some of the monomials counted by these coefficients do not belong to $B_1\{u\}$. The solution is to iterate the decomposition, by Lemma~\ref{lem:decomp}, on each of the chosen binomial coefficients. Hence, we have:
\begin{equation}
\label{eq:excountstrong1_1-2}
\begin{aligned}
&\binom{10}{3}=\underline{\binom{9}{2}}+\underline{\binom{8}{2}}+\underline{\binom{7}{2}}+\underline{\binom{6}{2}}+\binom{5}{2}+\binom{4}{2}+\binom{3}{2}+\binom{2}{2},\\
&\binom{9}{3}=\underline{\binom{8}{2}}+\underline{\binom{7}{2}}+\underline{\binom{6}{2}}+\binom{5}{2}+\binom{4}{2}+\binom{3}{2}+\binom{2}{2}.
\end{aligned}
\end{equation}
Now, we can repeat the previous procedure considering the second index of $u$, $i_2=5$. From (\ref{eq:excountstrong1_1-2})$_1$, considering the meaning of the coefficients, we must take the first $i_2-(1+t)+1=4$ binomials coefficients: $\binom{9}{2}+\binom{8}{2}+\binom{7}{2}+\binom{6}{2}$.

From (\ref{eq:excountstrong1_1-2})$_2$, we must take the first $5-(2+1)+1=3$ binomials coefficients: $\binom{8}{2}+\binom{7}{2}+\binom{6}{2}$.
%We can observe that this last number is included in the previous one.
The sum of all the underlined binomial coefficients in (\ref{eq:excountstrong1_1-2}) is the bound $c_2$.

Furthermore, we must consider the third index of $u$, $i_3=8$. From the first selection of coefficients in (\ref{eq:excountstrong1_1-2})$_1$, we compute further decompositions from which to take a decreasing number of binomial coefficients at each step, starting from $i_3-2-2t+2=8-2=6$ (indeed, the maximum value is when $s_1=s_2=1$). So, we obtain:
\begin{equation}\label{eq:excountstrong1_3}
\begin{aligned}
\binom{9}{2}&={\bf \binom{8}{1}+\binom{7}{1}+\binom{6}{1}+\binom{5}{1}+\binom{4}{1}+\binom{3}{1}}+\binom{2}{1}+\binom{1}{1},\\
\binom{8}{2}&={\bf \binom{7}{1}+\binom{6}{1}+\binom{5}{1}+\binom{4}{1}+\binom{3}{1}}+\binom{2}{1}+\binom{1}{1},\\
\binom{7}{2}&={\bf \binom{6}{1}+\binom{5}{1}+\binom{4}{1}+\binom{3}{1}}+\binom{2}{1}+\binom{1}{1},\\
\binom{6}{2}&={\bf \binom{5}{1}+\binom{4}{1}+\binom{3}{1}}+\binom{2}{1}+\binom{1}{1}.
\end{aligned}
\end{equation}

For this path, the procedure can no longer be iterated, in fact these coefficients count all the monomials $w$ with $j_1=1$, $j_2\leq 5$, $j_3\leq 8$ and $j_4\leq 11$. All the highlighted binomial coefficients in (\ref{eq:excountstrong1_3}) give a contribute to $c_3$.
So, we have the following partial value:
\[
(8+7+6+5+4+3)+(7+6+5+4+3)+(6+5+4+3)+(5+4+3)=88.
\]

Now, we need to consider the second selections of binomials taken from (\ref{eq:excountstrong1_1-2})$_2$. In such a case, we can repeat exactly the previous reasoning, getting:
\begin{equation}\label{eq:excountstrong1_4}
\begin{aligned}
\binom{8}{2}&={\bf \binom{7}{1}+\binom{6}{1}+\binom{5}{1}+\binom{4}{1}+\binom{3}{1}}+\binom{2}{1}+\binom{1}{1},\\
\binom{7}{2}&={\bf \binom{6}{1}+\binom{5}{1}+\binom{4}{1}+\binom{3}{1}}+\binom{2}{1}+\binom{1}{1},\\
\binom{6}{2}&={\bf \binom{5}{1}+\binom{4}{1}+\binom{3}{1}}+\binom{2}{1}+\binom{1}{1}.
\end{aligned}
\end{equation}

Therefore, the number of monomials of $B_1\{u\}$ with $j_1=2$ is
\[
(7+6+5+4+3)+(6+5+4+3)+(5+4+3)=55.
\]

Finally, we have all the information to compute the cardinality of the $1$-strongly stable set generated by $u$: $c_3=|B_1\{u\}|=88+55=143$.
 		
The following scheme summarizes the reasoning made up to now for counting the monomials of $B_1\{ x_{2}x_{5}x_{8}x_{11} \}$:
\begin{align*}
\tbinom{11}{4}= &\underline{\tbinom{10}{3}}+\underline{\tbinom{9}{3}}+\tbinom{8}{3} + \tbinom{7}{3}+\tbinom{6}{3}+\tbinom{5}{3}+\tbinom{4}{3}+\tbinom{3}{3}\\
& \tbinom{10}{3}=
\begin{aligned}[t] & \underline{\tbinom{9}{2}} + \underline{\tbinom{8}{2}} + \underline{\tbinom{7}{2}} + \underline{\tbinom{6}{2}} + \tbinom{5}{2}+\tbinom{4}{2}+\tbinom{3}{2}+\tbinom{2}{2}\\
& \tbinom{9}{2}=
\begin{aligned}[t] & {\bf \tbinom{8}{1}+\tbinom{7}{1}+\tbinom{6}{1}+\tbinom{5}{1}+\tbinom{4}{1}+\tbinom{3}{1}}+\tbinom{2}{1}+\tbinom{1}{1}\ \rightarrow \boxed{33}\\
& \tbinom{8}{2}=
\begin{aligned}[t] & {\bf \tbinom{7}{1}+\tbinom{6}{1}+\tbinom{5}{1}+\tbinom{4}{1}+\tbinom{3}{1}}+\tbinom{2}{1}+\tbinom{1}{1}\ \rightarrow \boxed{25}\\
& \tbinom{7}{2}=
\begin{aligned}[t] & {\bf \tbinom{6}{1}+\tbinom{5}{1}+\tbinom{4}{1}+\tbinom{3}{1}}+\tbinom{2}{1}+\tbinom{1}{1}\ \rightarrow \boxed{18}\\
& \tbinom{6}{2}={\bf \tbinom{5}{1}+\tbinom{4}{1}+\tbinom{3}{1}}+\tbinom{2}{1}+\tbinom{1}{1}\ \rightarrow \boxed{12}
\end{aligned}
\end{aligned}
\end{aligned}\\
& \tbinom{9}{3}=
\begin{aligned}[t] & \underline{\tbinom{8}{2}} + \underline{\tbinom{7}{2}} + \underline{\tbinom{6}{2}} + \tbinom{5}{2}+\tbinom{4}{2}+\tbinom{3}{2}+\tbinom{2}{2}\\
& \tbinom{8}{2}=
\begin{aligned}[t] & {\bf \tbinom{7}{1}+\tbinom{6}{1}+\tbinom{5}{1}+\tbinom{4}{1}+\tbinom{3}{1}}+\tbinom{2}{1}+\tbinom{1}{1}\ \rightarrow \boxed{25}\\
& \tbinom{7}{2}=
\begin{aligned}[t] & {\bf \tbinom{6}{1}+\tbinom{5}{1}+\tbinom{4}{1}+\tbinom{3}{1}}+\tbinom{2}{1}+\tbinom{1}{1}\ \rightarrow \boxed{18}\\
& \tbinom{6}{2}= {\bf \tbinom{5}{1}+\tbinom{4}{1}+\tbinom{3}{1}}+\tbinom{2}{1}+\tbinom{1}{1}\ \rightarrow \boxed{12}
\end{aligned}
\end{aligned}
\end{aligned}
\end{aligned}
\end{align*}
We can note that the binomial coefficients in bold  are precisely those described by the Formula~(\ref{eq:countstrong3_a}). Moreover, their number is $\C_3(6,4,2)$ where the arguments are the maximum values of the indexes $j_1$, $j_2$ and $j_3$ taken in reverse order. In the Remark~\ref{rem:operator} we have seen that $\C_3(6,4,2)=30$. 
\end{example}

In order to point out the methodologies to compute the cardinality of $t$-strongly stable sets, we will consider the same monomial of Example~\ref{ex:countstrong1} but in a $2$-spread contest.

\begin{example}\label{ex:countstrong2}
Let $S=K[x_1,\ldots,x_{13}]$, $t=2$ and $u=x_{i_1}x_{i_2}x_{i_3}x_{i_4}=x_2x_5x_8x_{11}$. We want to compute $c_3=|B_2[x_1x_3x_5x_7,u]|=|B_2\{u\}|$. 

%be the monomial of $S$ of which we want to count the cardinality of the $2$-strongly stable set generated by $u$. The greatest $2$-spread monomial of $S$, with respect to $>_{\slex}$, is $x_1x_3x_5x_7$. So, we want to compute $c_3=|B_2[x_1x_3x_5x_7,u]|=|B_2\{u\}|$. 
		
As previously done, we consider the number of all $t$-spread monomials of $S$ of degree $4$: $c=|M_{11,4,2}|=\binom{11-3}{4}=\binom{8}{4}=70$.
		
In order to compute $c_3=|B_2\{u\}|$, we take the following binomial decomposition:
\begin{equation}\label{eq:excountstrong2_0}
\binom{8}{4}=\underline{\binom{7}{3}}+\underline{\binom{6}{3}}+\binom{5}{3}+\binom{4}{3}+\binom{3}{3}.
\end{equation}

With the same scheme used in Example~\ref{ex:countstrong1}, we start looking at the first index of $u$, $i_1=2$. So, we consider the first two binomial coefficients in (\ref{eq:excountstrong2_0}), \emph{i.e.}, $c_1=\binom{7}{3}+\binom{6}{3}$. Clearly, we must iterate the decomposition on each of the chosen binomial coefficients:
\begin{equation}\label{eq:excountstrong2_1-2}
\begin{aligned}
\binom{7}{3}&=\underline{\binom{6}{2}}+\underline{\binom{5}{2}}+\underline{\binom{4}{2}}+\binom{3}{2}+\binom{2}{2},\\
\binom{6}{3}&=\underline{\binom{5}{2}}+\underline{\binom{4}{2}}+\binom{3}{2}+\binom{2}{2}.
\end{aligned}
\end{equation}

Considering the second index of $u$, $i_2=5$, from (\ref{eq:excountstrong2_1-2})$_1$, we must take the first $i_2-1-t+1=3$ binomials coefficients: $\binom{6}{2}+\binom{5}{2}+\binom{4}{2}$.

From (\ref{eq:excountstrong2_1-2})$_2$, we must take the first $5-2-t+1=2$ binomials coefficients: $\binom{5}{2}+\binom{4}{2}$.
%Also in such a case, this last number is included in the previous one.
Also in this case, the sum of all the underlined binomial coefficients in (\ref{eq:excountstrong2_1-2}) is the bound $c_2$.
%Also in such a case, this last number is included in the previous one.

Finally, we look at the third index of $u$, $i_3=8$. The maximum number of binomial coefficients we must take into consideration %for the further decompositions
is $i_3-2-2t+2=4$. This value will decrease by $1$ for the next binomial coefficient, and so on. Hence, we have: 
\begin{equation}\label{eq:excountstrong2_3}
\begin{aligned}
\binom{6}{2}&={\bf \binom{5}{1}+\binom{4}{1}+\binom{3}{1}+\binom{2}{1}}+\binom{1}{1},\\
\binom{5}{2}&={\bf \binom{4}{1}+\binom{3}{1}+\binom{2}{1}}+\binom{1}{1},\\
\binom{4}{2}&={\bf \binom{3}{1}+\binom{2}{1}}+\binom{1}{1}.
\end{aligned}
\end{equation}

Finally, we have the following partial value of $c_3$:
\[
(5+4+3+2)+(4+3+2)+(3+2)=14+14=28.
\]

From (\ref{eq:excountstrong2_1-2})$_2$, we can do analogous operations:
\begin{equation}\label{eq:excountstrong2_4}
\begin{aligned}
\binom{5}{2}&={\bf \binom{4}{1}+\binom{3}{1}+\binom{2}{1}}+\binom{1}{1},\\
\binom{4}{2}&={\bf \binom{3}{1}+\binom{2}{1}}+\binom{1}{1}.
\end{aligned}
\end{equation}

In such a case, the number of monomials analyzed is
\[
(4+3+2)+(3+2)=14,
\]
and the cardinality of the $2$-spread strongly stable set generated by $u$ is $c_3=|B_2\{u\}|=28+14=42$. The procedure can be summarized as follows:
%\begin{figure}[H]
\begin{align*}
\tbinom{8}{4}= & \underline{\tbinom{7}{3}} + \underline{\tbinom{6}{3}} + \tbinom{5}{3}+\tbinom{4}{3}+\tbinom{3}{3}\\
& \tbinom{7}{3}=
\begin{aligned}[t] & \underline{\tbinom{6}{2}} + \underline{\tbinom{5}{2}} + \underline{\tbinom{4}{2}}+\tbinom{3}{2}+\tbinom{2}{2}\\
& \tbinom{6}{2}=
\begin{aligned}[t] & {\bf \tbinom{5}{1}+\tbinom{4}{1}+\tbinom{3}{1}+\tbinom{2}{1}}+\tbinom{1}{1}\ \rightarrow \boxed{14}\\
& \tbinom{5}{2}=
\begin{aligned}[t] & {\bf \tbinom{4}{1}+\tbinom{3}{1}+\tbinom{2}{1}}+\tbinom{1}{1}\ \rightarrow \boxed{\phantom{1}9}\\
& \tbinom{4}{2}= {\bf \tbinom{3}{1}+\tbinom{2}{1}}+\tbinom{1}{1}\ \rightarrow \boxed{\phantom{1}5}
\end{aligned}
\end{aligned}\\
& \tbinom{6}{3}=
\begin{aligned}[t] & \underline{\tbinom{5}{2}} + \underline{\tbinom{4}{2}} + \tbinom{3}{2}+\tbinom{2}{2}\\
& \tbinom{5}{2}=
\begin{aligned}[t] & {\bf \tbinom{4}{1}+\tbinom{3}{1}+\tbinom{2}{1}}+\tbinom{1}{1}\ \rightarrow \boxed{\phantom{1}9}\\
& \tbinom{4}{2}= {\bf \tbinom{3}{1}+\tbinom{2}{1}}+\tbinom{1}{1}\ \rightarrow \boxed{\phantom{1}5}
\end{aligned}
\end{aligned}
\end{aligned}
\end{align*}
%\caption{Counting monomials of $B_2\{ x_{2}x_{5}x_{8}x_{11} \}$.}
%\label{fig:excountstrong2_1}
%\end{figure}

Also in this case, the binomial coefficients in bold are those described by the Formula~(\ref{eq:countstrong3_a}). Similarly to Example~\ref{ex:countstrong1}, the number of these binomial coefficients is
\[
\C_3(4,3,2)=\C_2(4,3)+\C_2(3,2)=(4+3+2)+(3+2)=14,
\]
%Figure \ref{fig:carnumfig2} in Appendix \ref{App:B} shows the list of all monomials which come into play to determine $[v,u]$.
and the monomials in $B_2\{u\}$ turn out to be
\begin{align*}
\begin{array}{r}
x_{1}x_{3}x_{5}x_{7},\,x_{1}x_{3}x_{5}x_{8},\,x_{1}x_{3}x_{5}x_{9},\,x_{1}x_{3}x_{5}x_{10},\,x_{1}x_{3}x_{5}x_{11},\\
x_{1}x_{3}x_{6}x_{8},\,x_{1}x_{3}x_{6}x_{9},\,x_{1}x_{3}x_{6}x_{10},\,x_{1}x_{3}x_{6}x_{11},\\
x_{1}x_{3}x_{7}x_{9},\,x_{1}x_{3}x_{7}x_{10},\,x_{1}x_{3}x_{7}x_{11},\\
x_{1}x_{3}x_{8}x_{10},\,x_{1}x_{3}x_{8}x_{11},
\end{array} \rightarrow \boxed{14}\\
\begin{array}{r}
x_{1}x_{4}x_{6}x_{8},\,x_{1}x_{4}x_{6}x_{9},\,x_{1}x_{4}x_{6}x_{10},\,x_{1}x_{4}x_{6}x_{11},\\
x_{1}x_{4}x_{7}x_{9},\,x_{1}x_{4}x_{7}x_{10},\,x_{1}x_{4}x_{7}x_{11},\\
x_{1}x_{4}x_{8}x_{10},\,x_{1}x_{4}x_{8}x_{11},
\end{array} \rightarrow \boxed{\phantom{1}9}\\
\begin{array}{r}
x_{1}x_{5}x_{7}x_{9},\,x_{1}x_{5}x_{7}x_{10},\,x_{1}x_{5}x_{7}x_{11},\,x_{1}x_{5}x_{8}x_{10},\,x_{1}x_{5}x_{8}x_{11},
\end{array} \rightarrow \boxed{\phantom{1}5}\\
\begin{array}{r}
x_{2}x_{4}x_{6}x_{8},\,x_{2}x_{4}x_{6}x_{9},\,x_{2}x_{4}x_{6}x_{10},\,x_{2}x_{4}x_{6}x_{11},\\
x_{2}x_{4}x_{7}x_{9},\,x_{2}x_{4}x_{7}x_{10},\,x_{2}x_{4}x_{7}x_{11},\\
x_{2}x_{4}x_{8}x_{10},\,x_{2}x_{4}x_{8}x_{11},
\end{array} \rightarrow \boxed{\phantom{1}9}\\
\begin{array}{r}
x_{2}x_{5}x_{7}x_{9},\,x_{2}x_{5}x_{7}x_{10},\,x_{2}x_{5}x_{7}x_{11},\\
x_{2}x_{5}x_{8}x_{10},\,x_{2}x_{5}x_{8}x_{11}.
\end{array} \rightarrow \boxed{\phantom{1}5}
\end{align*}
%\caption{Monomials of $B_2\{ x_{2}x_{5}x_{8}x_{11} \}$.}
%\label{fig:excountstrong2_2}
\end{example}

The pseudocode in Algorithm~\ref{alg:countstrong} shows the implementation of the steps used in the proof of Theorem~\ref{thm:countstrong}.
%\begin{center}
%\scalebox{0.8}{
\begin{algorithm}%[H]
\scriptsize
\caption{Computation of the cardinality $B_t\{u\} \subset M_{\overline n,d,t}$}
\label{alg:countstrong}
\KwIn{Monomial $u$, positive integer $t$}
\KwOut{positive integer $c$}
\SetKw{Error}{error}
\Begin{
   \eIf{\texttt{isTSpread($u$,$t$)}}{
      $m \gets \max(u)$\;      
      $d \gets \deg(u)$\;
      $c \gets 0$\;      
      $p \gets d-1$\;
      \For{$r \gets 1$ \KwTo $p$}{
         $ind(r) \gets 1$\;      
      }     

      \While{$ind(1)\leq j_1$}{
         $c \gets c+m-(d-1)*(t-1)-ind(1)-\cdots-ind(p)$\;
         $ind(p) \gets ind(p)+1$\;
         %$ind(p) > j_p-(ind(1)+\cdots+ind(p)-p+p*t+1)+1$
         \While{$p>0$ and $ind(p) > i_p-ind(1)-\cdots-ind(p)+p*(1-t)$}{
            $ind(p) \gets 1$\;
            $p \gets p-1$\;
            $ind(p) \gets ind(p)+1$\;
         }
         $p \gets d-1$\;
      }   
   }{
      \Error{expected a $t$-spread monomial}\;      
   }    
\Return $c$\;
}
\end{algorithm}
%}
%\end{center}

The mechanism used in the algorithm fully exploits the theoretical result in (\ref{eq:countstrong3_a}). More precisely, a list of $d-1$ positive integer, initialized to $1$, acts as the multi-index $(s_1,\ldots,s_{d-1})$. At each step the achievement of the maximum value for the involved component is dynamically checked and the multi-index is updated.

%\subsection{Some examples}\label{sec:3_3}

\section{Observations and outlook}\label{sec:4}

The algorithmic construction and the examples in this paper have been tested using the package \texttt{TSpreadIdeals} running on \emph{Macaulay2} 1.81.
%created by using the \emph{Macaulay2} package \texttt{TSpreadIdeals}, and 

To the best of our knowledge, packages for managing classes of $t$-spread ideals have not been heretofore implemented. Most of the algorithms presented in the package are due to the author of this paper. All the methods discussed using pseudocode derive from results of the same authors that can be found in this or in other papers.

We are confident that the package \texttt{TSpreadIdeals} can be used for further studies and applications. In fact, the methods it contains could be useful for investigating the $t$-spread \emph{generic initial ideal} or in general to find alternative methods to compute the \emph{generic initial ideal} of a graded squarefree ideal \cite{AHH3}.

Moreover, we are working to optimize construction and counting methods for $t$-stable sets of monomials. As a first step, we plan to use the same approach for $t$-strongly stable sets, even if  
we are  aware that  for the $t$-stable case the \emph{linearization} could encounter some difficulties. 

We are confident that many other problems will arise around $t$-spread structures, and that, consequently, new implementations could be added to the package in order to improve it so providing new tools and functionalities.

%These problems are currently under investigation by the authors.

\appendix

\section{Package overview} \label{sec:5}
In this Section, we describe the main features of the \emph{Macaulay2} package \texttt{TSpreadIdeals}.

First, we present some algorithms implemented in the package not yet discussed in the previous sections. Then, we illustrate how to use the main functions of the package by applying them to the solution of appropriate examples.

\subsection{Package usage}\label{sec:5_1}

In this Section, we show some \emph{Macaulay2} instructions to manage $t$-spread monomials and to solve some related problems for the classes of $t$-lex and $t$-strongly stable segments or ideals.

In the first example, we display some auxiliary methods for $t$-spread monomials/structures, \emph{e.g.}, for verifying if a monomial/ideal is $t$-spread or for sieving $t$-spread monomials from a list of monomials. 

\begin{example}
Let $S=K[x_1,\ldots,x_{14}]$, $u=x_3x_7x_{10}x_{14}\in M_{14,4,3}$ and $v=x_1x_5x_9x_{13}\in M_{14,4,4}$. We can use the method \texttt{isTSpread} to check whether a monomial, a list of monomials or an ideal is $t$-spread. Moreover, given a list of monomials we can extract all the $t$-spread monomials of the list.
% , fontsize=small]
\begin{Verbatim}[xleftmargin=0.5cm]
i1 : loadPackage "TSpreadIdeals";
i2 : S=QQ[x_1..x_14];
i3 : u=x_3*x_7*x_10*x_14;
i4 : isTSpread(u,3)
o4 : true
i5 : isTSpread(u,4)
o5 = false
i6 : v=x_1*x_5*x_9*x_13;
i7 : isTSpread(v,4)
o7 : true
i8 : l={u,v}
o8 : {x_3x_7x_10x_14, x_1x_5x_9x_13}
o8 : List
i9 : isTSpread(l,4)
o9 = false
i10 : I=ideal l
o10 : ideal(x_3x_7x_10x_14, x_1x_5x_9x_13)
o10 : Ideal of S
i11 : IsTSpread(I,4)
o11 : false
i12 : tSpreadList(l,4)
o12 = {x_1x_5x_9x_13}
o12 : List
i13 : isTSpread(oo,4)
o13 : true
\end{Verbatim}

Furthermore, we are able to compute the $t$-shadow of a set of $t$-spread monomials (Definition~\ref{def:shad}). This is a useful tool for some constructions related to the $f_t$-vector of a $t$-spread ideal (see \cite{CAC}).
 
\begin{Verbatim}[xleftmargin=0.5cm]
i14 : tShadow(l,3)
o14 : {}
o14 : List
i15 : tShadow(l,2)
o15 : {x_1x_3x_5x_9x_13, x_1x_3x_7x_10x_14, x_1x_5x_7x_9x_13,
       x_1x_5x_9x_11x_13, x_3x_5x_7x_10x_14, x_3x_7x_10x_12x_14}
o15 : List
\end{Verbatim}
\end{example}

Now, we illustrate how some examples of this paper have been obtained by using suitable functions of the package.  

\begin{example}(see Example~\ref{ex:countlex})\\
Let $S=K[x_1,\ldots,x_{11}]$, $t=3$ and $u=x_2x_6x_{10}\in M_{11,3,3}$. We want to compute $L_t\{u\}$ and its cardinality.

To obtain the results, one can use the function \texttt{tLexMon(u,t)} that returns the list of the monomials of $L_\texttt{t}\{\texttt{u}\}$. In detail, the monomials are constructed using the function \texttt{tNextMon(w,t)} iteratively (see Proposition~\ref{prop:tnextlex} and Algorithm~\ref{alg:tnextlex}). A single use of it will return the greatest monomial less than \texttt{u}, in $M_{n,d,t}$, with respect to $>_\slex$. So, to compute $L_\texttt{t}\{\texttt{u}\}$ the routine must start from the greatest monomial of $M_{n,d,t}$ until  \texttt{u} is reached.
The method \texttt{countTLexMon(u,t)} computes the cardinality of $L_t\{u\}$ by using the result of Theorem~\ref{thm:countlex} (see Algorithm~\ref{alg:countlex}).

\begin{Verbatim}[xleftmargin=0.5cm]
i1 : loadPackage "TSpreadIdeals";
i2 : S=QQ[x_1..x_11];
i3 : u=x_2*x_6*x_10;
i4 : l=tLexMon(u,3)
o4 = {x_1x_4x_7, x_1x_4x_8, x_1x_4x_9, ..., x_2x_6x_9, x_2x_6x_10}
o4 : List
i5 : #l
o5 = 21
i6 : countTLexMon(u,3)
o6 : 21
\end{Verbatim}

To verify that a set of monomials is a $t$-lex segment, the method \texttt{isTLexSeg} can be used. This method gets a list of monomials \texttt{l} and verifies if it coincides with the $t$-lex segment containing the maximum and the minimum of the list \texttt{l}.
In such a case, we can note that the segment in \texttt{l} is $L_3\{\texttt{u}\}$, \emph{i.e.}, an initial $t$-lex segment. It can be obtained using a more general method and fixing the starting monomial.  

\begin{Verbatim}[xleftmargin=0.5cm]
i7 : isTLexSeg l
o7 = true
i8 : ini=product for i to ((degree u)#0-1) list S_(i*t)
o8 = x_1x_4x_7
o8 : S
i9 : m=tLexSeg(ini,u,3)
o9 = {x_1x_4x_7, x_1x_4x_8, x_1x_4x_9, ..., x_2x_6x_9, x_2x_6x_10}
o9 : List
i10 : l==m
o10 : true
\end{Verbatim}

%We note that the Algorithm~\ref{alg:tnextlex} is very fast: it manipulates only the indexes of the given monomial to build the monomial in output. However, in a particular case in which we have about thirty thousands of monomials belongs to $L_\texttt{t}\{\texttt{u}\}$, the method takes about one minute to give the final list. We have not be able to obtains the result by using classic methods.\\
%Moreover, if one is only interested in the cardinality of the list, then the method to compute it (without construction) takes only few seconds to give the result. 
\end{example}

Examples~\ref{ex:countstrong1} and \ref{ex:countstrong2} differ from each other only for the $t$-spread context. We have proposed them both to highlight the differences in building and counting the monomials. %We describe the instruction in one session by only changing the value of $t$.
%We describe the instructions to solve both once, applied only to the second one.     

\begin{example}(see Examples~\ref{ex:countstrong1}, \ref{ex:countstrong2})\\
Let $S=K[x_1,\ldots,x_{13}]$ and $u=x_{i_1}x_{i_2}x_{i_3}x_{i_4}=x_2x_5x_8x_{11}$. We show the instructions to compute 
$B_t\{u\}$ and its cardinality for both values $t=1,2$.

We will use the function \texttt{tStronglyStableMon(u,t)} to obtain the list of the monomials of $B_\texttt{t}\{\texttt{u}\}$.
This method calls  \texttt{tStronglyStableSeg(ini,u,t)} (see Proposition~\ref{prop:tnextstrong} and Algorithm~\ref{alg:tnextstrong}), where \texttt{ini} is the greatest monomial of $M_{n,d,t}$. Indeed, we have already observed that this monomial always belongs to $B_\texttt{t}\{\texttt{u}\}$.
The function \texttt{countTStronglyStableMon(u,t)} gives the cardinality of $B_\texttt{t}\{\texttt{u}\}$ (see Theorem~\ref{thm:countstrong} and Algorithm~\ref{alg:countstrong}).
  
\begin{Verbatim}[xleftmargin=0.5cm]
i1 : loadPackage "TSpreadIdeals";
i2 : S=QQ[x_1..x_13];
i3 : u=x_2*x_5*x_8*x_11;
i4 : l_1=tStronglyStableMon(u,1)
o4 = {x_1x_2x_3x_4, x_1x_2x_3x_5, x_1x_2x_3x_6, ..., x_2x_5x_8x_10,
      x_2x_5x_8x_11}
o4 : List
i5 : #l_1
o5 = 143
i6 : countTStronglyStableMon(u,1)
o6 : 143
\end{Verbatim}

As in the previous example, the method \texttt{isTStronglyStableSeg(l,t)} checks if a set of monomials is a $t$-strongly stable segment. This method gets a list of monomials \texttt{l} and verifies if it coincides with the $t$-strongly stable segment containing the maximum and the minimum of the list \texttt{l}.
In the case of the example, we can note that the segment in \texttt{l} is $B_t\{\texttt{u}\}$, \emph{i.e.}, an initial $t$-strongly stable segment.

\begin{Verbatim}[xleftmargin=0.5cm]
i7 : isTStronglyStableSeg l_1
o7 = true
i8 : ini=product for i to ((degree u)#0-1) list S_(i*t)
o8 = x_1x_2x_3x_4
o8 : S
i9 : m_1=tStronglyStableSeg(ini,u,1)
o9 = {x_1x_2x_3x_4, x_1x_2x_3x_5, x_1x_2x_3x_6, ..., x_2x_5x_8x_10,
      x_2x_5x_8x_11}
o9 : List
i10 : l_1==m_1
o10 : true
\end{Verbatim}

For the example with $t=2$, we continue the previous session of \emph{Macaulay2} assigning the new value of $t$. The instructions are the same as before. 

\begin{Verbatim}[xleftmargin=0.5cm]
i11 : l_2=tStronglyStableMon(u,2)
o11 = {x_1x_3x_5x_7, x_1x_3x_5x_8, x_1x_3x_5x_9, ..., x_2x_5x_8x_10,
       x_2x_5x_8x_11}
o11 : List
i12 : #l_2
o12 : 42
i13 : countTStronglyStableMon(u,2)
o13 = 42
i14 : isTStronglyStableSeg l_2
o14 = true
i15 : ini=product for i to ((degree u)#0-1) list S_(i*t)
o16 = x_1x_3x_5x_7
o16 : S
i17 : m_2=tStronglyStableSeg(ini,u,2)
o17 = {x_1x_3x_5x_7, x_1x_3x_5x_8, x_1x_3x_5x_9, ..., x_2x_5x_8x_10,
      x_2x_5x_8x_11}
o17 : List
i18 : l_2==m_2
o18 : true
\end{Verbatim}

The construction of the $t$-strongly stable set $B_t\{u\}$ (Algorithm~\ref{alg:tnextstrong}) is very fast: it manipulates only the indexes of the given monomial to build the monomial in output. For instance, let $S=K[x_1,\ldots,x_{30}]$ and $u=x_3x_8x_{12}x_{18}x_{23}x_{28}\in M_{30,6,2}$. In such a particular case, we have just over thirty thousands monomials belonging to $B_\texttt{t}\{\texttt{u}\}$. The method \texttt{tStronglyStableMon} takes about one minute to give the final list. We have not be able to obtain the result by using classical  methods. Moreover, if one is only interested to the cardinality of the list (without construction), then the method \texttt{countTStronglyStableMon} takes only few seconds to give the result.
\end{example}

Now, we show the instructions to solve the Problem~\ref{prob:Betti}. In such a case, we illustrate the steps to develop an example already done in the paper \cite{ACF2}.

\begin{example}(see \cite[Example~5.1]{ACF2})\\
Let $S=K[x_1,\ldots,x_{25}]$, $t=3$ and $r=4$. We want to verify if the list of pairs $\C=\big\{(k_1,\ell_1),(k_2,\ell_2),(k_3,\ell_3),(k_4,\ell_4)\big\}=\big\{(6,2),(5,4),(4,5),(3,7)\big\}$ and the list $a=(a_1,a_2,a_3,a_4)=(2,1,3,2)$ represent the positions and the values of the extremal Betti numbers of a $t$-spread ideal.

To solve the problem, we use the function \texttt{tExtremalBettiMonomials} that returns, if possible, the list of the basic monomials involved in the characterization of the extremal Betti numbers in the $t$-spread context.
%This function also use the Algorithms \ref{alg:tnextAkl}, \ref{alg:countbasic} and \ref{alg:minbshad}.\\
Furthermore, in order to determine the smallest $t$-strongly stable ideal with the requested extremal Betti numbers we can use the function \texttt{tStronglyStableIdeal}.
  
\begin{Verbatim}[xleftmargin=0.5cm]
i1 : loadPackage "TSpreadIdeals";
i2 : S=QQ[x_1..x_25];
i3 : pairs={{6,2},{5,4},{4,5},{3,7}};
i4 : a={2,1,3,2};
i5 : bm=tExtremalBettiMonomials(S,pairs,a,t)
o5 : {x_1x_10, x_2x_10, x_3x_6x_9x_15,
      x_3x_6x_10x_13x_17, x_3x_6x_10x_14x_17, x_3x_6x_11x_14x_17,
      x_3x_7x_10x_13x_16x_19x_22, x_4x_7x_10x_13x_16x_19x_22}
o5 = List
i6 : I=tStronglyStableIdeal(ideal bm,t)
o6 : {x_1x_4, x_1x_5, x_1x_6, x_1x_7, x_1x_8, x_1x_9, x_1x_10, x_2x_5,
      x_2x_6, x_2x_7, x_2x_8, x_2x_9, x_2x_10, x_3x_6x_9x_12,
      x_3x_6x_9x_13, x_3x_6x_9x_14, x_3x_6x_9x_15, x_3x_6x_10x_13x_16,
      x_3x_6x_10x_13x_17, x_3x_6x_10x_14x_17, x_3x_6x_11x_14x_17,
      x_3x_7x_10x_13x_16x_19x_22, x_4x_7x_10x_13x_16x_19x_22}
o6 : Ideal of S
\end{Verbatim}

We can use the function \texttt{tExtremalBettiCorners} to verify if the ideal $I$ has the extremal Betti numbers in the given positions, that is, it has the corners coinciding with the given pairs. This method implements the computation of $\ds(I)$, the degree-sequence of $I$ (see (\ref{eq:degseq2})).  
%This method use a generalization of the \emph{degree sequence} (see \cite{AC7}) using the characterization of the extremal Betti numbers of a $t$-spread ideals (see \cite{AC8}).

\begin{Verbatim}[xleftmargin=0.5cm]
i7 : tExtremalBettiCorners(I,t)
o7 : {(6,2), (5,4), (4,5), (3,7)}
o7 = List
i8 : minimalBettiNumbersIdeal(I)
o8 :          0  1   2   3  4  5 6
     total : 23 77 117 100 51 15 2
       2   : 13 42  70  70 42 14 2
       3   :  -  -   -   -  -  - -
       4   :  4 14  20  15  6  1 -
       5   :  4 15  21  13  3  - -
       6   :  -  -   -   -  -  - -
       7   :  2  6   6   2  -  - -
o8 : BettyTally
\end{Verbatim}
 \end{example}

Finally, we illustrate some methods of the package to approach a generalization of the \emph{Kruskal-Katona}'s theorem to $t$-spread ideals, as shown in \cite{CAC}.

\begin{example}(see \cite[Example~2.4]{CAC} and \cite[Example~3.11]{CAC})\\
Let $S=K[x_1,\ldots,x_8]$ and let 
\[
I=(x_1x_3x_5, x_1x_3x_6, x_1x_3x_7, x_1x_3x_8, x_1x_4x_6, x_1x_4x_7, x_1x_4x_8, x_2x_4x_6, x_2x_4x_7, x_2x_4x_8)
\]
be the $2$-spread ideal.

The function \texttt{fTVector(I,t)} is able to compute the $f_\texttt{t}$-vector of a $t$-spread ideal. Moreover, the method \texttt{tLexIdeal(I,t)} computes, if it exists, the $t$-lex ideal that share the same $f_t$-vector of $I$. \\ The overloaded function, \texttt{tLexIdeal(S,f,t)}, allows to compute the $t$-lex ideal of \texttt{S} whose $f_t$-vector is \texttt{f}, if it exists.  
  
\begin{Verbatim}[xleftmargin=0.5cm]
i1 : loadPackage "TSpreadIdeals";
i2 : S=QQ[x_1..x_8];
i3 : I=ideal(x_1*x_3*x_5, x_1*x_3*x_6, x_1*x_3*x_7, x_1*x_3*x_8,
             x_1*x_4*x_6, x_1*x_4*x_7, x_1*x_4*x_8, x_2*x_4*x_6,
             x_2*x_4*x_7, x_2*x_4*x_8);
i4 : isTLexIdeal(I,2)
o4 : false
i5 : fTVector(I,2)
o5 = {1, 8, 21, 10, 0}
o5 = List
i6 : IL=tLexIdeal(I,2)
o6 : {x_1x_3x_5, x_1x_3x_6, x_1x_3x_7, x_1x_3x_8, x_1x_4x_6, x_1x_4x_7,
      x_1x_4x_8, x_1x_5x_7, x_1x_5x_8, x_1x_6x_8, x_2x_4x_6x_8}
o6 : Ideal of S
i7 : fTVector(IL,2)
o7 = {1, 8, 21, 10, 0}
o7 = List
\end{Verbatim}

In the next example, we define a sequence of nonnegative integers and check if it satisfies the conditions of the \emph{KK} theorem. The function \texttt{tMacaulayExpansion} returns the binomial expansion defined in \cite[Definition~3.5]{CAC}. This method is used in the implementation of \texttt{isFTVector(S,f,t)} that verifies whether a sequence of integer is the $f_\texttt{t}$-vector of a \texttt{t}-strongly stable ideal of \texttt{S}.

\begin{Verbatim}[xleftmargin=0.5cm]
i1 : loadPackage "TSpreadIdeals";
i2 : S=QQ[x_1..x_12];
i3 : f={1,12,50,20,15}
o3 = {1, 12, 50, 20, 15}
o3 : List
i4 : tMacaulayExpansion(f_1,12,1,1,Shift=>true)
o4 = {{12,2}}
o4 : List
i5 : solveBinomialExpansion oo
o5 : 66
i6 : solveBinomialExpansion tMacaulayExpansion(f_2,12,2,1,Shift=>true)
o6 : 130
i7 : solveBinomialExpansion tMacaulayExpansion(f_3,12,3,1,Shift=>true)
o7 : 15
i8 : solveBinomialExpansion tMacaulayExpansion(f_4,12,4,1,Shift=>true)
o8 = 6
i9 : ifFTVector(S,f,1)
o9 = true
i10 : I=tLexIdeal(S,f,1)
o10 = {x_1x_2, x_1x_3, x_1x_4, ... , x_7x_9x_10x_11x_12, x_8x_9x_10x_11x_12}
o10 : Ideal of S 
i11 : numgens I
o11 : 132
i12 : fTVector(I,1)==f
o12 : true
i13 : ifFTVector(S,f,2)
o13 : false
i14 : tLexIdeal(S,f,2)
stdio:9:1:(3): error: expected a valid ft-vector
i15 : solveBinomialExpansion tMacaulayExpansion(f_3,12,3,2,Shift=>true)
o15 : 5
i16 : f_4<oo
o16 : false
\end{Verbatim}
\end{example}

\subsection{List of the main functions of \texttt{TSpreadIdeals}}\label{sec:5_2}

\hspace{-1.4cm}
\begin{tabular}{p{0.45\textwidth}p{0.70\textwidth}}
\texttt{isTSpread(u,t)} or \texttt{isTSpread(l,t)} or \texttt{isTSpread(I,t)} & Check whether a monomial, a list of monomials or an ideal is \texttt{t}-spread.\\[3pt]
\texttt{tSpreadList(l,t)} & Give all the \texttt{t}-spread monomials of the list \texttt{l}.\\[3pt]
\texttt{tSpreadIdeal(I,t)} & Give the \texttt{t}-spread ideal generated by all the \texttt{t}-spread monomials of the ideal \texttt{I}.\\[3pt]
\texttt{tNexMon(u,t)} \newline option \texttt{FixedMax=$>$Boolean} & Give the \texttt{t}-lex successor of \texttt{u}. If \texttt{FixedMax} is \texttt{true} then it gives the successor of \texttt{u} in $A^t(k,\ell)$.\\[3pt]
\texttt{tLastMon(u,gap,t)} \newline option \texttt{MaxInd=$>$ZZ} & Give the last \texttt{t}-spread monomial of the shadow of $B_t\{$\texttt{u}$\}$ with respect to $>_\slex$. If \texttt{MaxInd} is nonnegative it gives $\min\BShad(u)_{(k,\ell)}$.\\[3pt]
\texttt{tShadow(u,t)} or \texttt{tShadow(l,t)} & Give the \texttt{t}-shadow of the \texttt{t}-spread monomial \texttt{u} or of all the \texttt{t}-spread monomials in the list \texttt{l}.\\[3pt]
\texttt{tLexSeg(v,u,t)} & Give the \texttt{t}-lex segment $L_t[\texttt{v, u}]$.\\[3pt]
\texttt{isTLexSeg(l,t)} & Check whether a \texttt{t}-spread set \texttt{l} is a $t$-lex segment.\\[3pt]
\texttt{tLexMon(u,t)} & Give the initial \texttt{t}-lex segment $L_\texttt{t}\{\texttt{u}\}$.\\[3pt]
\texttt{countTLexMon(u,t)} \newline option \texttt{FixedMax=$>$Boolean}& Give the cardinality of $L_\texttt{t}\{\texttt{u}\}$. If \texttt{FixedMax} is \texttt{true} then it gives the cardinality $L_\texttt{t}\{\texttt{u}\} \cap A^t(k,\ell)$.\\[3pt]
\texttt{tVeroneseSet(Ring,d,t)} & Give the set of all the \texttt{t}-spread monomials of degree \texttt{d} of \texttt{S}.\\[3pt]
\texttt{tVeroneseIdeal(Ring,d,t)} & Give the ideal generated by all the \texttt{t}-spread monomials of degree \texttt{d} of \texttt{S}.\\[3pt]
\texttt{tPascalIdeal(Ring,t)} & Give the \texttt{t}-spread Pascal ideal of \texttt{S}.\\[3pt]
\texttt{fTVector(I,t)} & Give the $f_\texttt{t}$ vector of the \texttt{t}-spread ideal \texttt{I}.\\[3pt]
\texttt{tMacaulayExpansion(a,n,d,t)} \newline option \texttt{Shift$=>$Boolean}& Give the \texttt{t}-spread \texttt{d}-th Macaulay expansion of \texttt{a} on \texttt{n} indeterminates. If \texttt{Shift} is \texttt{true} it applies the KK's operator.\\[3pt]
\texttt{isFTVector(S,f,t)} & Check whether the sequence \texttt{f} is the $f_\texttt{t}$ vector of a suitable \texttt{t}-strongly stable ideal of \texttt{S}.\\[3pt]
%\texttt{solveBinomialExpansion(l)} & Give the sum the the binomial coefficients listed in \texttt{l}\\
\texttt{tLexIdeal(I,t)} or \texttt{tLexIdeal(S,f,t)} & Give the \texttt{t}-lex ideal sharing the same $f_\texttt{t}$-vector of \texttt{I} or the \texttt{t}-lex ideal of \texttt{S} with \texttt{f} as $f_\texttt{t}$-vector.\\[3pt]
\texttt{isTLexIdeal(I,t)} & Check whether the \texttt{t}-spread ideal \texttt{I} is a \texttt{t}-lex ideal.\\[3pt]
%\texttt{isTCompletelyLex(I,t)} & Check whether the \texttt{t}-spread ideal \texttt{I} is completely lex\\
%\texttt{tSigmaOperatorMon(v,t)} & Give the \texttt{t}-spread monomial obtained by applying the\\
%& $\sigma$-operator to \texttt{v}\\
%\texttt{tSigmaOperatorIdeal(I,t)} & Give the ideal generated by all the \texttt{t}-spread monomials\\
%& obtained by applying the $\sigma$-operator to the generators of \texttt{I}\\
\texttt{tStronglyStableSeg(v,u,t)} & Give the \texttt{t}-strongly stable segment of monomials between the monomials \texttt{v} and \texttt{u}.\\[3pt]
\texttt{isTStronglyStableSeg(l,t)} & Check whether the list of monomials \texttt{l} is a \texttt{t}-strongly stable segment.\\[3pt]
\texttt{tStronglyStableMon(u,t)} & Give the \texttt{t}-strongly stable set of monomials generated by \texttt{u}.\\[3pt]
\texttt{countTStronglyStableMon(u,t)} & Give the cardinality of $ B_\texttt{t}\{\texttt{u}\}$.\\[3pt]
\texttt{tStronglyStableIdeal(I,t)} & Give the smallest \texttt{t}-strongly stable ideal containing \texttt{I}.\\[3pt]
\texttt{isTStronglyStableIdeal(I,t)} & Check whether the ideal \texttt{I} is \texttt{t}-strongly stable.\\[3pt]
%\texttt{tStableMon(v,t)} & Give the \texttt{t}-stable set of monomials generated by \texttt{v}\\
%\texttt{countTStableMon(v,t)} & Give the number of \texttt{t}-strongly stable monomials generated by \texttt{v}\\
%\texttt{tStableSeg(v,u,t)} & Give the \texttt{t}-strongly stable segment of monomials between the monomials \texttt{v} and \texttt{u}\\
%\texttt{tStableIdeal(I,t)} & Give the smallest \texttt{t}-stable ideal containing \texttt{I}\\
%\texttt{isTStableIdeal(I,t)} & Check whether the ideal \texttt{I} is \texttt{t}-stable\\
%\texttt{supportIndex(I)} & Give the support index of the \texttt{t}-squarefree ideal \texttt{I}\\
\texttt{tExtremalBettiCorners(I,t)} & Give the list of the corners of \texttt{I}.\\[3pt]
%\texttt{countTBasicMonomials(u,t)} & Give the number of monomials of $A^t(k,\ell)$ greater than \texttt{u}\\
\texttt{tExtremalBettiMonomials(S,corn,a,t)} & Give the basic \texttt{t}-spread monomials for the extremal Betti numbers assigned positions as well as values.\\
\end{tabular}

\bibliographystyle{abbrv}
\bibliography{Amata}{}

\begin{thebibliography}{10}

\bibitem{AC8}
L.~Amata and M.~Crupi.
\newblock Extremal {B}etti numbers of t-spread strongly stable ideals.
\newblock {\em Mathematics}, 7(8):1--16, 2019.

\bibitem{AC7}
L.~Amata and M.~Crupi.
\newblock On the extremal {B}etti numbers of squarefree monomial ideals.
\newblock {\em {I}nt. {E}lectron. {J}. {A}lgebra}, 30:168--202, 2021.

\bibitem{ACF2}
L.~Amata, M.~Crupi, and A.~Ficarra.
\newblock A numerical characterization of the extremal {B}etti numbers of
  $t$--spread strongly stable ideals.
\newblock {\em {J}ournal of {A}lgebraic {C}ombinatorics}, 55(3):891--918, 2022.

\bibitem{ACF3}
L.~Amata, M.~Crupi, and A.~Ficarra.
\newblock Projective dimension and {C}astelnuovo-{M}umford regularity of
  t-spread ideals.
\newblock {\em Internat. {J}. {A}lgebra {C}omput.}, 32(4):837--858, 2022.

\bibitem{ACF1}
L.~Amata, M.~Crupi, and A.~Ficarra.
\newblock Upper bounds for extremal {B}etti numbers of $t$--spread strongly
  stable ideals.
\newblock {\em {B}ull. {M}ath. {S}oc. {S}ci. {M}ath. {R}oumanie ({N.S.})},
  65(113)(1):13--34, 2022.

\bibitem{AEL}
C.~Andrei, V.~Ene, and B.~Lajmiri.
\newblock Powers of t-spread principal borel ideals.
\newblock {\em Archiv der Mathematik}, 112(6):587--597, Jun 2019.

\bibitem{CAC}
C.~Andrei-Ciobanu.
\newblock Kruskal-{K}atona theorem for t-spread strongly stable ideals.
\newblock {\em {B}ull. {M}ath. {S}oc. {S}ci. {M}ath. {R}oumanie ({N.S.})},
  62(110)(2):107--122, 2019.

\bibitem{AHH2}
A.~Aramova, J.~Herzog, and T.~Hibi.
\newblock Squarefree lexsegment ideals.
\newblock {\em Mathematische Zeitschrift}, 228:353--378, 06 1998.

\bibitem{AHH3}
A.~Aramova, J.~Herzog, and T.~Hibi.
\newblock Shifting operations and graded {B}etti numbers.
\newblock {\em Journal of Algebraic Combinatorics}, 12(3):207--222, Nov 2000.

\bibitem{BCP}
D.~Bayer, H.~Charalambous, and S.~Popescu.
\newblock Extremal {B}etti numbers and applications to monomial ideals.
\newblock {\em Journal of Algebra}, 221:497--512, 1999.

\bibitem{DHQ}
R.~Dinu, J.~Herzog, and A.~A. Qureshi.
\newblock Restricted classes of {V}eronese type ideals and algebras.
\newblock {\em Internat. {J}. {A}lgebra {C}omput.}, 31(1):173--197, 2021.

\bibitem{Ei}
D.~Eisenbud.
\newblock {\em Commutative algebra, with a view toward algebraic geometry},
  volume 150 of {\em Graduate Texts in Mathematics}.
\newblock Springer-Verlag, 1995.

\bibitem{EHQ}
V.~Ene, J.~Herzog, and A.~A. Qureshi.
\newblock t-spread strongly stable monomial ideals*.
\newblock {\em Communications in Algebra}, 47(12):5303--5316, 2019.

\bibitem{GDS}
D.~R. Grayson and M.~E. Stillman.
\newblock Macaulay2, a software system for research in algebraic geometry.
\newblock Available at \url{http://www.math.uiuc.edu/Macaulay2/}, 2020.

\bibitem{JT}
J.~Herzog and T.~Hibi.
\newblock {\em Monomial ideals}, volume 260 of {\em Graduate texts in
  mathematics}.
\newblock Springer-Verlag London, 1 edition, 2011.

\bibitem{MS}
E.~Miller and B.~Sturmfels.
\newblock {\em Combinatorial Commutative Algebra}, volume 227 of {\em Graduate
  Texts in Mathematics}.
\newblock Springer-Verlag New York, 1 edition, 2005.

\end{thebibliography}

\end{document}